\documentclass[11pt]{amsart}

\usepackage[a4paper,margin=1.12in]{geometry}
\usepackage{amsmath,amssymb,mathtools}
\usepackage{microtype}
\usepackage[hidelinks]{hyperref}

\usepackage{color}

\newtheorem{theorem}{Theorem}[section]
\newtheorem{proposition}[theorem]{Proposition}
\newtheorem{lemma}[theorem]{Lemma}
\newtheorem{corollary}[theorem]{Corollary}
\theoremstyle{remark}
\newtheorem{remark}[theorem]{Remark}

\newcommand{\M}{\mathcal M}
\newcommand{\N}{\mathcal N}
\newcommand{\A}{\mathcal A}
\newcommand{\K}{\mathcal K}
\newcommand{\Tr}{\operatorname{Tr}}
 
\newcommand{\one}{\mathbf 1}
\newcommand{\C}{\mathbb C}

\title[Three-point trace positivity and interpolating metrics]
{Three-Point Trace Positivity:
	Schatten Norm Compression and Interpolating Metrics}

\author{Trung Dung Vuong}

\address{
	High School for the Gifted, VNUHCM,
	153 Nguyen Chi Thanh, An Dong Ward,
	Ho Chi Minh City, Vietnam
}

\address{
	Vietnam National University Ho Chi Minh City,
	Linh Xuan Ward,
	Ho Chi Minh City, Vietnam
}

\email{vtdung@ptnk.edu.vn}

\author{Hiroyuki Osaka}

\address{
Department of Mathematics,
Ritsumeikan University,
Kusatsu, Shiga 525-8577, Japan
}

\email{osaka@se.ritsumei.ac.jp}


\begin{document}
	
\begin{abstract}
	Let $(\N,\nu)$ be a finite tracial von Neumann algebra and set
	$\K_r(P,Q)=\nu((P^{1/2}QP^{1/2})^r)$. We prove that
	$$
	[\K_r(P_i,P_j)]_{i,j=1}^3\ge0,
	\qquad 1/2\le r\le1.
	$$
	The range $1/2<r<1$ is new even for matrices, and order three is sharp
	for $1/2\le r<1$. This single mechanism resolves two problems. First, if
	$(\mathcal R,\omega)$ is finite tracial and
	$H=[H_{ij}]_{i,j=1}^3\in M_3(\mathcal R)_+$, then
	$$
	[\omega(|H_{ij}|^p)]_{i,j=1}^3\ge0,
	\qquad 1\le p\le2,
	$$
	settling the Lin--van den Driessche conjecture on Schatten norm
	compression. Second, for a semifinite tracial von Neumann algebra
	$(\M,\tau)$, define
	$$
	d_t(A,B)^2=\frac{\tau(A)+\tau(B)}2-
	\|B^{t/4}A^{t/4}\|_{2/t}^{2/t}.
	$$
	We prove that $d_t$ is a complete metric on $L^1(\M,\tau)_+$ for
	$1\le t\le2$, including the previously open range $1<t<2$ between the
	Hellinger and Bures endpoints, and that
	$$
	d_t\text{ is a metric on }L^1(\M,\tau)_+
	\quad\Longleftrightarrow\quad
	t\in[1,2]\ \text{or}\ \M\text{ is abelian}.
	$$
	Its topology and Cauchy sequences agree with those of $L^1$. The same
	classification holds for unital tracial $C^*$-algebras and, with the known
	$t=0$ obstruction, resolves Problem~2 of Kom\'alovics and Moln\'ar.
\end{abstract}

	\maketitle
	
\section{Introduction}
\label{sec:introduction}

The main result of this paper is a three-point positivity theorem for a
nonlinear trace kernel. It provides a common mechanism for two
apparently different problems: a conjecture on Schatten norm compression of
positive $3\times3$ block matrices, and the previously open metricity range
between the Hellinger and Bures endpoints. We state the three principal
results first and then place them in context.

Let $(\M,\tau)$ be a von Neumann algebra with a faithful normal semifinite
trace; neither finiteness nor $\sigma$-finiteness is assumed. For
$0<p<\infty$, let $L^p(\M,\tau)$ denote the associated space of
$\tau$-measurable operators, with the standard product conventions of \cite{FackKosaki}.
We use the generalized H\"older inequality in
\cite[Theorem~4.2(i)]{FackKosaki}, including the range
of exponents below one.  If
$A,B\in L^1(\M,\tau)_+$ and $t>0$, then
$B^{t/4}A^{t/4}\in L^{2/t}(\M,\tau)$, and we set
\begin{align}
Q_t(A,B)
&:=\bigl\|B^{t/4}A^{t/4}\bigr\|_{2/t}^{\,2/t}\notag\\
&=\tau\!\left(
\left(A^{t/4}B^{t/2}A^{t/4}\right)^{1/t}
\right),\label{eq:Qt-intro}\\
d_t(A,B)^2&:=\frac{\tau(A)+\tau(B)}2-Q_t(A,B).
\label{eq:dt-intro}
\end{align}
Adjoint invariance of generalized singular values and H\"older's inequality
show that $Q_t$ is symmetric and
$0\le Q_t(A,B)\le\sqrt{\tau(A)\tau(B)}$. Thus $d_t$ is well defined as the
nonnegative square root for every $t>0$, even when $2/t<1$. When necessary, we write $Q_t^\tau$ and $d_t^\tau$ to specify
the underlying trace.

Let $\N$ be a von Neumann algebra equipped with a faithful
normal finite trace $\nu$. For $r>0$, define
\begin{equation}\label{eq:kernel}
	\K_r^\nu(P,Q)
	:=
	\nu\!\left((P^{1/2}QP^{1/2})^r\right),
	\qquad P,Q\in\N_+.
\end{equation}

Our main trace-kernel result is the following.

\begin{theorem}
\label{thm:three-point}
For every $1/2\le r\le1$ and $P_1,P_2,P_3\in\N_+$,
\begin{equation}\label{eq:three-point}
\bigl[\K_r^\nu(P_i,P_j)\bigr]_{i,j=1}^3\ge0.
\end{equation}
\end{theorem}

The endpoint $r=1$ is the Gram matrix
$[\nu(P_iP_j)]_{i,j=1}^3$ in $L^2(\N,\nu)$. For matrix density operators,
the endpoint $r=1/2$ is the three-state root-fidelity matrix. Its positivity
was asserted in \cite[Proposition~1]{FannesDeMeloRogaZyczkowski}; Lin and
van den Driessche subsequently pointed out a serious gap in the published
proof \cite[p.~835]{LinDriessche}. The endpoint follows independently from
Drury's trace-norm theorem \cite[Corollary~1.3]{Drury}, and is also recovered
by our argument. To the best of our knowledge, the intermediate range
$1/2<r<1$ is new even for matrices.

Theorem~\ref{thm:three-point} is genuinely a finite-order statement, not an
assertion that $\K_r^\nu$ is a positive definite kernel on arbitrary finite
sets. The explicit four-pure-state construction in
Remark~\ref{rem:three-point-scope} gives a negative eigenvalue for every
$1/2\le r<1$; hence order three is sharp throughout this nontrivial range.
At $r=1$, by contrast, $\K_1^\nu$ is a Gram kernel at arbitrary order. Our
proof of \eqref{eq:three-point} is not an interpolation between the two
endpoints; it is based on a one-versus-many trace inequality that controls
all sign patterns in a three-dimensional quadratic form.

The first application is a three-block Schatten compression theorem.

\begin{theorem}
\label{thm:tracial-block-compression}
Let $(\mathcal R,\omega)$ be a von Neumann algebra with a faithful
normal finite trace $\omega$, and let $H=[H_{ij}]_{i,j=1}^3\in M_3(\mathcal R)_+$.
Then, for every $1\le p\le2$,
\begin{equation}\label{eq:block-compression}
\bigl[\omega(|H_{ij}|^p)\bigr]_{i,j=1}^3\ge0.
\end{equation}
\end{theorem}

For matrix blocks, \eqref{eq:block-compression} says precisely that
$[\|H_{ij}\|_p^p]_{i,j=1}^3$ is positive semidefinite. This resolves the
conjecture of Lin and van den Driessche
\cite[Conjecture~5.1]{LinDriessche}. Conversely, in the matrix setting,
Theorem~\ref{thm:three-point} follows from this compression statement by
applying it, with $p=2r$, to the positive block matrix
$H=[P_i^{1/2}P_j^{1/2}]_{i,j=1}^3$. Indeed, $H=X^*X$ for
$X=[P_1^{1/2}\ P_2^{1/2}\ P_3^{1/2}]$, and
$\|H_{ij}\|_{2r}^{2r}=\K_r^{\Tr}(P_j,P_i)
=\K_r^{\Tr}(P_i,P_j)$. Thus the two matrix statements are equivalent. The
endpoint $p=1$ is Drury's theorem \cite{Drury}, whereas $p=2$ follows from
Marcus and Watkins \cite[Theorem~1]{MarcusWatkins}; see also \cite{LiZhang}
for related unitarily invariant norm formulations. In the scalar-block case,
the theorem of FitzGerald and Horn \cite{FitzGeraldHorn} shows that the
three-block restriction is essential: for four scalar blocks,
counterexamples exist for every nonintegral $1<p<2$. Related higher-block
norm-compression statements also fail \cite{Audenaert,LinDriessche}.

The second application concerns the metric geometry of the entire
semifinite $L^1$-cone.
If $\A$ is a unital $C^*$-algebra with a faithful tracial positive linear
functional $\tau$, the trace expressions in \eqref{eq:Qt-intro} and
\eqref{eq:dt-intro}, interpreted by continuous functional calculus in $\A$,
define $Q_t^\tau$ and $d_t^\tau$ on $\A_+$.

The metricity, completeness, and sharpness conclusions are summarized next.

\begin{theorem}
\label{thm:classification}
\begin{enumerate} 
  \item[\textup{(i)}] For every $1\le t\le2$, $d_t$ is a complete
	metric on $L^1(\M,\tau)_+$. The topology induced by $d_t$ is the
	$L^1$-topology, and a sequence in $L^1(\M,\tau)_+$ is
	$d_t$-Cauchy if and only if it is $L^1$-Cauchy. 

\item[\textup{(ii)}] For every $t>0$,
\begin{equation}\label{eq:sharp-classification}
d_t\text{ is a metric on }L^1(\M,\tau)_+
\quad\Longleftrightarrow\quad
t\in[1,2]\ \text{or}\ \M\text{ is abelian}.
\end{equation}
If $\M$ is nonabelian and $t\notin[1,2]$, failure of the triangle inequality
already occurs on bounded positive elements with finite-trace support; when
$\tau(\one)<\infty$, the counterexample may be chosen positive invertible.

\item[\textup{(iii)}] If $\A$ is a unital $C^*$-algebra with a faithful
tracial positive linear functional $\tau$, then, for every $t>0$,
$d_t^\tau$ is a metric on $\A_+$ if and only if $t\in[1,2]$ or $\A$ is
abelian.
\end{enumerate}
\end{theorem}

The endpoint metrics in Theorem~\ref{thm:classification}\textup{(i)} are well established:
$d_1(A,B)=2^{-1/2}\|A^{1/2}-B^{1/2}\|_2$ is the Hellinger metric up to
normalization, while $Q_2$ is root fidelity. On the unit-trace density space,
$d_2$ agrees with the Bures metric of
\cite[Definition~1.1]{FarenickRahaman}.
Our family $d_t^\tau$ has the normalization used for $d_p^\tau$
in \cite{KomalovicsMolnar}; their Bures distance $d_B^\tau$
equals $\sqrt2\,d_2^\tau$. Bures metricity,
data processing, and topological questions in tracial operator algebras have
been studied in \cite{AguilarBeheraHjelmborgEtAl,AguilarBeheraOmlandWu,
FarenickJaquesRahaman,FarenickRahaman}. These works concern the
Bures endpoint. The new metric content here is the range $1<t<2$, its
derivation from Theorem~\ref{thm:three-point}, and the sharp classification
on the full, generally unbounded, semifinite $L^1$-cone; our argument also
identifies the induced topology and Cauchy structure.

Kom\'alovics and Moln\'ar \cite[Problem~2]{KomalovicsMolnar} asked whether,
for every $n\ge2$, $d_t^{\Tr}$ is a metric for $1<t<2$, fails to be one for
$0\le t<1$, and is a metric for $1<t<2$ on every unital tracial
$C^*$-algebra. At $t=0$ the power expression is understood through its
symmetrized Lie--Trotter, or log-Euclidean, limit.

Several boundary results were known. On rank-one projections, Kom\'alovics
and Moln\'ar proved that $d_t^{\Tr}$ is a metric exactly for $0<t\le2$;
their Proposition~8 supplies a $2\times2$ triangle defect for every $t>2$
\cite[Proposition~8]{KomalovicsMolnar}. At $t=0$, Bhatia, Gaubert, and Jain
exhibited a positive definite $2\times2$ triangle defect for the same
log-Euclidean distance, up to the factor $\sqrt2$
\cite[Section~2]{BhatiaGaubertJain}. More recently, Zhang proved the
underlying nonnegativity estimate for every $t>0$ and found a
positive definite $2\times2$ violation at $t=1/2$
\cite[Theorem~3.3 and Proposition~4.2]{Zhang}; the rest of $0<t<1$ remained
open.

For matrices, with $z=1/t$, the squared distance
$d_t^{\Tr}(A,B)^2$ coincides with the symmetric
$\alpha$-$z$-Bures--Wasserstein divergence at $\alpha=1/2$
introduced in \cite{DinhLeVoVuong}.  Related work of the first author gives conditional
quasi-triangle and stability estimates under spectral or local hypotheses
\cite{VuongQuantitative}, but leaves exact global metricity for $1/2<z<1$
open. The present theorem proves the exact global inequality on the full
semifinite cone and determines every admissible parameter; those related
estimates imply neither this conclusion nor the three-point positivity above.

The common proof mechanism begins with a tracial Carlen--Lieb convexity
principle \cite{CarlenLieb}, proved directly on $L^2(\N,\nu)$ by Effros's
operator-perspective method \cite{Effros}, using the Hansen--Pedersen Jensen
inequality \cite{HansenPedersen}; compare \cite{Bekjan}. It yields a
transported Minkowski inequality and a
one-versus-many Cauchy--Schwarz estimate, which controls all sign patterns in
a three-dimensional quadratic form and proves
Theorem~\ref{thm:three-point}. No matrix approximation is used. Factoring a
positive block matrix as $[X_i^*X_j]$ then identifies its entrywise
$p$-compression with Theorem~\ref{thm:three-point} at $r=p/2$, proving
Theorem~\ref{thm:tracial-block-compression}.

For $1\le t\le2$, the substitution $r=1/t$, $P_i=A_i^{t/2}$ realizes every
three-point subset of a finite tracial positive cone in a real Hilbert space,
with Euclidean distances $\sqrt2d_t$. Spectral compression of each fixed
triple and $L^1$-continuity of $Q_t$ extend the triangle inequality to the
semifinite cone without a global finite-trace exhaustion. Endpoint comparison
then gives separation, the $L^1$-topology, coincidence of Cauchy sequences,
and completeness.

For sharpness, strict $2\times2$ defects for $0<t<1$ come from Clarkson's
inequality, while the case $t>2$ uses the rank-one obstruction of
\cite[Proposition~8]{KomalovicsMolnar}. Every nonabelian semifinite algebra
has a finite-trace corner containing a unital copy of $M_2(\C)$, so these defects prove
\eqref{eq:sharp-classification}; in the abelian case every $d_t$ reduces to
the Hellinger metric. The tracial GNS envelope and positive Kaplansky density
give the $C^*$-algebra assertion. Together with the known logarithmic
obstruction \cite{BhatiaGaubertJain}, this resolves
\cite[Problem~2]{KomalovicsMolnar}.

We also record data processing for $d_t$ from the known
$\alpha$-$z$ inequality \cite{HiaiJencova}, together with metric completions
of finite-support and $C^*$-algebraic cores.

Sections~\ref{sec:finite-tracial} and \ref{sec:block} prove
Theorems~\ref{thm:three-point} and
\ref{thm:tracial-block-compression}. Sections~\ref{sec:semifinite} and
\ref{sec:classification} prove the three parts of
Theorem~\ref{thm:classification}. Section~\ref{sec:zero-endpoint} treats the
logarithmic endpoint, while Section~\ref{sec:cores} develops the further
consequences and comments on a possible type~III extension.
	\section{Three-point trace kernels in tracial von Neumann algebras}
	\label{sec:finite-tracial}
	
Throughout this section, $(\N,\nu)$ denotes an arbitrary von Neumann
algebra equipped with a faithful normal finite trace; thus
$\nu(\one)<\infty$, but the normalization $\nu(\one)=1$ is not
required. The algebra $\N$ is not assumed to be hyperfinite. 
	Both $Q_t$ and $d_t$ are computed with respect to $\nu$;
	we omit the superscript $\nu$ in this section.

	The analytic starting point is a direct realization of the required
	Carlen--Lieb convexity \cite{CarlenLieb} through commuting left and right
	multiplication operators on the tracial $L^2$-space. The trace $\nu$ need not be normalized, and $\N^{++}$
	denotes the cone of positive invertible elements of $\N$.
	Throughout this section, we write $\K_r=\K_r^\nu$.
	For any other faithful trace $\omega$ the notation $\K_r^\omega$
	denotes the same kernel with $\nu$ replaced by $\omega$.
By the trace property,
\begin{equation}\label{eq:kernel-basic}
	\K_r(P,Q)=\K_r(Q,P),
	\qquad
	\K_r(P,P)=\nu(P^{2r}).
\end{equation}
	Moreover,
	\begin{equation}\label{eq:kernel-Lp}
		\K_r(P,Q)=\|Q^{1/2}P^{1/2}\|_{2r}^{2r}.
	\end{equation}
	Hence, for $r\ge1/2$, noncommutative H\"older gives
	\begin{equation}\label{eq:pairwise-CS}
		\K_r(P,Q)^2\le\K_r(P,P)\K_r(Q,Q).
	\end{equation}
	
	We begin with the tracial form of Carlen--Lieb convexity needed below.
	
	\begin{proposition}
		\label{prop:tracial-CL}
		Let $1\le s\le2$ and $L\in\N$. Then
		\begin{equation}\label{eq:CL-functional}
			F_{s,L}(W):=\nu\!\left((L^*W^sL)^{1/s}\right),
			\qquad W\in\N_+,
		\end{equation}
		is convex and positively homogeneous. Consequently,
		\begin{equation}\label{eq:CL-subadditivity}
			F_{s,L}\!\left(\sum_{j=1}^m c_jW_j\right)
			\le\sum_{j=1}^m c_jF_{s,L}(W_j),
			\qquad c_j\ge0.
		\end{equation}
	\end{proposition}
	
	\begin{proof}
		If $s=1$, then
		\[
		F_{1,L}(W)
		=
		\nu(L^*WL)
		=
		\nu(LL^*W),
		\]
		which is linear in $W$. We therefore assume throughout the rest of
		the proof that
		$
		1<s\le2.
		$
		
		We divide the proof into three steps.
		
		\smallskip
		\noindent
		\emph{Step 1: an Ando-type jointly convex functional.}
Equip $L^2(\N,\nu)$ with the inner product
\[
\langle Y,Z\rangle_2:=\nu(Y^*Z).
\]
Since $\nu(\one)<\infty$ and $L\in\N$, we have
\[
\|L\|_2^2
=
\nu(L^*L)
\le
\|L\|_\infty^2\nu(\one)
<\infty.
\]
Thus $L\in L^2(\N,\nu)$.

For $W,X\in\N^{++}$, define the left- and
right-multiplication operators on $L^2(\N,\nu)$ by
\[
\lambda_W(Y):=WY,
\qquad
\rho_X(Y):=YX,
\qquad Y\in L^2(\N,\nu).
\]
		These operators are bounded and positive. Indeed, for every
$Y\in L^2(\N,\nu)$,
\[
\|\lambda_WY\|_2
\le \|W\|_\infty\|Y\|_2,
\qquad
\|\rho_XY\|_2
\le \|X\|_\infty\|Y\|_2,
\]
while
\begin{align*}
	\langle Y,\lambda_WY\rangle_2
	&=
	\nu(Y^*WY)
	=
	\|W^{1/2}Y\|_2^2,\\
	\langle Y,\rho_XY\rangle_2
	&=
	\nu(Y^*YX)
	=
	\|YX^{1/2}\|_2^2.
\end{align*}
Since $W,X\in\N^{++}$, there exist
$\varepsilon_W,\varepsilon_X>0$ such that
\[
W\ge\varepsilon_W\one,
\qquad
X\ge\varepsilon_X\one.
\]
It follows that
\[
\lambda_W\ge\varepsilon_W I_{L^2},
\qquad
\rho_X\ge\varepsilon_X I_{L^2},
\]
so both operators are invertible. Here and below, real powers of
positive invertible elements and operators are defined by functional
calculus. Its compatibility with left and right multiplication gives
\[
\lambda_W^a=\lambda_{W^a},
\qquad
\rho_X^a=\rho_{X^a},
\qquad a\in\mathbb R.
\]
Every left-multiplication operator commutes with every
right-multiplication operator:
\[
\lambda_W\rho_X(Y)
=
WYX
=
\rho_X\lambda_W(Y).
\]
In particular,
\begin{equation}\label{eq:cross-commutation}
	[\lambda_{W_i},\rho_{X_j}]=0,
	\qquad i,j\in\{1,2\}.
\end{equation}

We next recall the operator-perspective argument in the precise
Hilbert-space form needed here. Let
\[
f_s(u):=u^s,\qquad u>0.
\]
This function is operator convex for $1\le s\le2$.
Suppose that $P_i,Q_i$, $i=1,2$, are bounded positive
invertible operators on a Hilbert space such that
\[
[P_i,Q_j]=0,
\qquad i,j\in\{1,2\}.
\]
For $0<\theta<1$, put
\[
P_\theta:=\theta P_1+(1-\theta)P_2,
\qquad
Q_\theta:=\theta Q_1+(1-\theta)Q_2,
\]
and define
\[
C_1
:=
\theta^{1/2}Q_1^{1/2}Q_\theta^{-1/2},
\qquad
C_2
:=
(1-\theta)^{1/2}Q_2^{1/2}Q_\theta^{-1/2}.
\]
Then
\[
C_1^*C_1+C_2^*C_2=\one
\]
and
\[
Q_\theta^{-1/2}P_\theta Q_\theta^{-1/2}
=
\sum_{i=1}^2
C_i^*
\bigl(Q_i^{-1/2}P_iQ_i^{-1/2}\bigr)
C_i.
\]
Effros formulated the corresponding perspective theorem in the
matrix setting \cite[Theorem~2.2]{Effros}. Since its proof is based on
the Hansen--Pedersen--Jensen inequality, we record the same argument
on the possibly infinite-dimensional Hilbert space under
consideration.

The Hansen--Pedersen operator Jensen inequality
\cite[Theorem~2.1(ii)]{HansenPedersen} gives
\[
f_s\!\left(
Q_\theta^{-1/2}P_\theta Q_\theta^{-1/2}
\right)
\le
\sum_{i=1}^2
C_i^*
f_s\!\left(
Q_i^{-1/2}P_iQ_i^{-1/2}
\right)
C_i.
\]
Taking the congruence by $Q_\theta^{1/2}$ preserves the operator
order.  Write
\[
\alpha_1:=\theta,
\qquad
\alpha_2:=1-\theta.
\]
The cross-commutation assumptions imply
\[
[P_\theta,Q_\theta]=0,
\qquad
[P_i,Q_i]=0,\quad i=1,2.
\]
Since $f_s(u)=u^s$, functional calculus therefore gives
\[
Q_\theta^{1/2}
f_s\!\left(
Q_\theta^{-1/2}P_\theta Q_\theta^{-1/2}
\right)
Q_\theta^{1/2}
=
P_\theta^sQ_\theta^{\,1-s}.
\]
Moreover, from
\[
C_i=\alpha_i^{1/2}Q_i^{1/2}Q_\theta^{-1/2}
\]
we obtain
\[
Q_\theta^{1/2}C_i^*
=
\alpha_i^{1/2}Q_i^{1/2},
\qquad
C_iQ_\theta^{1/2}
=
\alpha_i^{1/2}Q_i^{1/2},
\]
and hence
\[
\begin{aligned}
	&Q_\theta^{1/2}C_i^*
	f_s\!\left(
	Q_i^{-1/2}P_iQ_i^{-1/2}
	\right)
	C_iQ_\theta^{1/2}\\
	&\qquad=
	\alpha_iQ_i^{1/2}
	\left(Q_i^{-1/2}P_iQ_i^{-1/2}\right)^s
	Q_i^{1/2}
	=
	\alpha_iP_i^sQ_i^{\,1-s}.
\end{aligned}
\]
Consequently,
\begin{equation}\label{eq:commuting-perspective}
	P_\theta^sQ_\theta^{\,1-s}
	\le
	\theta P_1^sQ_1^{\,1-s}
	+
	(1-\theta)P_2^sQ_2^{\,1-s}.
\end{equation}
Thus the Hansen--Pedersen argument yields the dimension-free
Hilbert-space form of Effros's commuting-perspective inequality.
No finite-dimensional approximation is involved.
Now let $W_i,X_i\in\N^{++}$, $i=1,2$. Applying
\eqref{eq:commuting-perspective} with
\[
P_i=\lambda_{W_i},
\qquad
Q_i=\rho_{X_i},
\]
is legitimate by \eqref{eq:cross-commutation}. Since the assignments
$W\mapsto\lambda_W$ and $X\mapsto\rho_X$ are linear, we obtain
\begin{align*}
	&
	\lambda_{\theta W_1+(1-\theta)W_2}^{\,s}
	\rho_{\theta X_1+(1-\theta)X_2}^{\,1-s}
	\\
	&\qquad\le
	\theta\lambda_{W_1}^{\,s}\rho_{X_1}^{\,1-s}
	+
	(1-\theta)
	\lambda_{W_2}^{\,s}\rho_{X_2}^{\,1-s}.
\end{align*}
Hence
\[
(W,X)\longmapsto
\lambda_W^s\rho_X^{\,1-s}
\]
is jointly operator convex on
$\N^{++}\times\N^{++}$.

Evaluating this operator inequality at the fixed vector
$L\in L^2(\N,\nu)$, we obtain
\begin{align*}
	\left\langle
	L,\lambda_W^s\rho_X^{\,1-s}L
	\right\rangle_2
	&=
	\left\langle
	L,W^sLX^{1-s}
	\right\rangle_2\\
	&=
	\nu(L^*W^sLX^{1-s}).
\end{align*}
Consequently,
\begin{equation}\label{eq:Ando-form}
	(W,X)
	\longmapsto
	\nu(L^*W^sLX^{1-s})
\end{equation}
is jointly convex on
$\N^{++}\times\N^{++}$.

It remains to allow $W$ to be noninvertible. Let
$W_i\in\N_+$, $X_i\in\N^{++}$, $i=1,2$, and apply the
preceding inequality to
\[
(W_i+\varepsilon\one,X_i),
\qquad \varepsilon>0.
\]
Since
\[
\theta(W_1+\varepsilon\one)
+(1-\theta)(W_2+\varepsilon\one)
=
\theta W_1+(1-\theta)W_2+\varepsilon\one,
\]
letting $\varepsilon\downarrow0$ and using the norm continuity
of $W\mapsto W^s$ on $\N_+$ proves that
\eqref{eq:Ando-form} is jointly convex on
\[
\N_+\times\N^{++}.
\]
		\smallskip
		\noindent
	\emph{Step 2: a tracial variational formula.}
	We claim that, for every $C\in\N_+$,
	\begin{equation}\label{eq:root-variational}
		\nu(C^{1/s})
		=
		\frac1s
		\inf_{B\in\N^{++}}
		\left\{
		\nu(CB^{1-s})+(s-1)\nu(B)
		\right\}.
	\end{equation}
	
	Fix $B\in\N^{++}$ and define
	\[
	F_B(A):=\nu(A^sB^{1-s}),
	\qquad A\in\N_+.
	\]
	Here $B$ is kept fixed, and all derivatives below are taken with
	respect to the variable $A$.  The linear functional
	\[
	\ell_B(Y):=\nu(YB^{1-s})
	\]
	is positive.  Indeed, for $Y\in\N_+$, traciality gives
	\[
	\ell_B(Y)
	=
	\nu\!\left(
	B^{(1-s)/2}YB^{(1-s)/2}
	\right)
	\geq0.
	\]
	Since $u\mapsto u^s$ is operator convex, it follows that
	$F_B=\ell_B\circ(u\mapsto u^s)$ is convex.
	
	Because $B$ is an interior point of $\N_+$, the
	supporting-hyperplane inequality at $A=B$ gives, for every
	$A\in\N_+$,
	\begin{equation}\label{eq:FB-supporting}
		F_B(A)
		\geq
		F_B(B)
		+
		DF_B(B)[A-B].
	\end{equation}

	We next compute the derivative in \eqref{eq:FB-supporting}.  Let
	$p_s(u):=u^s$.  Since $B\in\N^{++}$, holomorphic functional calculus
	gives, for $H\in\N_{\rm sa}$,
	\[
	D(p_s)_B[H]
	=
	\frac{1}{2\pi i}
	\int_\Gamma
	z^s(z\one-B)^{-1}H(z\one-B)^{-1}\,dz,
	\]
	where $\Gamma\subset\C\setminus(-\infty,0]$ is a positively oriented
	contour surrounding $\sigma(B)$.  Put
	\[
	R_z:=(z\one-B)^{-1}.
	\]
	Since $B$ is fixed in the definition of $F_B$,
	\[
	DF_B(B)[H]
	=
	\nu\!\left(D(p_s)_B[H]B^{1-s}\right).
	\]
	The trace $\nu$ is a bounded linear functional because
	$\nu(\one)<\infty$.  We may therefore pass it through the contour
	integral.  Using traciality and the fact that $R_z$ commutes with
	$B^{1-s}$, we obtain
	\begin{align*}
		DF_B(B)[H]
		&=
		\frac{1}{2\pi i}
		\int_\Gamma
		z^s\nu(R_zHR_zB^{1-s})\,dz\\
		&=
		\nu\!\left(
		HB^{1-s}
		\frac{1}{2\pi i}
		\int_\Gamma z^sR_z^2\,dz
		\right).
	\end{align*}
	Cauchy's differentiation formula gives
	\[
	\frac{1}{2\pi i}
	\int_\Gamma z^s(z\one-B)^{-2}\,dz
	=
	sB^{s-1}.
	\]
	Consequently,
	\begin{align*}
		DF_B(B)[H]
		&=
		\nu\!\left(HB^{1-s}sB^{s-1}\right)\\
		&=
		s\nu(H).
	\end{align*}
	
 	Since $A-B\in\N_{\rm sa}$, the preceding derivative formula gives
	\[
	DF_B(B)[A-B]
	=
	s\nu(A-B).
	\]
	Substituting this identity and $F_B(B)=\nu(B)$ into
	\eqref{eq:FB-supporting}, we obtain
	\begin{align*}
		\nu(A^sB^{1-s})
		&\geq
		\nu(B)+s\nu(A-B)\\
		&=
		\nu(B)+s\nu(A)-s\nu(B)\\
		&=
		s\nu(A)+(1-s)\nu(B).
\end{align*}
	Equivalently,
	\begin{equation}\label{eq:tracial-Young}
		\nu(A^sB^{1-s})+(s-1)\nu(B)
		\geq
		s\nu(A).
	\end{equation}
	
	Set $A=C^{1/s}$, inequality
	\eqref{eq:tracial-Young} becomes
	\[
	\nu(CB^{1-s})+(s-1)\nu(B)
	\geq
	s\nu(C^{1/s})
	\]
	for every $B\in\N^{++}$.  Hence
	\[
	\frac1s
	\inf_{B\in\N^{++}}
	\left\{
	\nu(CB^{1-s})+(s-1)\nu(B)
	\right\}
	\geq
	\nu(C^{1/s}).
	\]
	
	Suppose first that $C$ is invertible.  Equality is attained by choosing
	$
	B=C^{1/s}.
	$
	
	It remains to consider a possibly noninvertible $C$.  Put
	\[
	A:=C^{1/s},
	\qquad
	B_\varepsilon:=A+\varepsilon\one\in\N^{++},
	\qquad \varepsilon>0.
	\]
	Because $A$ and $B_\varepsilon$ commute, functional calculus gives
	\[
	A^sB_\varepsilon^{1-s}
	=
	A^s(A+\varepsilon\one)^{1-s}.
	\]
	For $a\geq0$, the scalar inequality
	\[
	0
	\leq
	a-a^s(a+\varepsilon)^{1-s}
	\leq
	(s-1)\varepsilon
	\]
	implies, again by functional calculus,
	\[
	0
	\leq
	A-A^s(A+\varepsilon\one)^{1-s}
	\leq
	(s-1)\varepsilon\one.
	\]
	Applying $\nu$ and using $\nu(\one)<\infty$, we obtain
	\[
	0
	\leq
	\nu(A)
	-
	\nu\!\left(A^s(A+\varepsilon\one)^{1-s}\right)
	\leq
	(s-1)\varepsilon\nu(\one).
	\]
	Therefore,
	\[
	\nu\!\left(A^s(A+\varepsilon\one)^{1-s}\right)
	\longrightarrow
	\nu(A)
	\]
	as $\varepsilon\downarrow0$.  Moreover,
	\[
	\nu(B_\varepsilon)
	=
	\nu(A)+\varepsilon\nu(\one)
	\longrightarrow
	\nu(A).
	\]
	Since $C=A^s$, it follows that
	\[
	\nu(CB_\varepsilon^{1-s})
	+
	(s-1)\nu(B_\varepsilon)
	\longrightarrow
	s\nu(A)
	=
	s\nu(C^{1/s}).
	\]
	Thus the infimum in \eqref{eq:root-variational} is at most
	$s\nu(C^{1/s})$.  Together with the opposite inequality proved above,
	this establishes \eqref{eq:root-variational} for every $C\in\N_+$.
	
		\smallskip
		\noindent
		\emph{Step 3: convexity of $F_{s,L}$.}
		Applying \eqref{eq:root-variational} to
		\[
		C=L^*W^sL
		\]
		gives
		\begin{equation}\label{eq:CL-variational}
			F_{s,L}(W)
			=
			\frac1s
			\inf_{X\in\N^{++}}
			\Psi(W,X),
		\end{equation}
		where
		\[
		\Psi(W,X)
		:=
		\nu(L^*W^sLX^{1-s})+(s-1)\nu(X).
		\]
		The first term is jointly convex in $(W,X)$ by
		\eqref{eq:Ando-form}, while the second term is linear in $X$.
		Therefore $\Psi$ is jointly convex.
		
		We spell out why taking the infimum over $X$ preserves convexity.
		Let $W_1,W_2\in\N_+$, $0\le\theta\le1$, and choose arbitrary
		$X_1,X_2\in\N^{++}$. Since
		\[
		X_\theta:=\theta X_1+(1-\theta)X_2\in\N^{++},
		\]
		joint convexity gives
		\begin{align*}
			sF_{s,L}\bigl(\theta W_1+(1-\theta)W_2\bigr)
			&\le
			\Psi\bigl(
			\theta W_1+(1-\theta)W_2,X_\theta
			\bigr)\\
			&\le
			\theta\Psi(W_1,X_1)
			+
			(1-\theta)\Psi(W_2,X_2).
		\end{align*}
		Taking the infimum independently over $X_1$ and $X_2$, and using
		\eqref{eq:CL-variational}, yields
		\[
		F_{s,L}\bigl(\theta W_1+(1-\theta)W_2\bigr)
		\le
		\theta F_{s,L}(W_1)
		+
		(1-\theta)F_{s,L}(W_2).
		\]
		Thus $F_{s,L}$ is convex.
		
		Finally, for every $c\ge0$,
		\begin{align*}
			F_{s,L}(cW)
			&=
			\nu\!\left(
			\bigl(L^*(cW)^sL\bigr)^{1/s}
			\right)\\
			&=
			\nu\!\left(
			\bigl(c^sL^*W^sL\bigr)^{1/s}
			\right)\\
			&=
			cF_{s,L}(W).
		\end{align*}
		Hence $F_{s,L}$ is positively homogeneous.
		
		To prove \eqref{eq:CL-subadditivity}, put
		\[
		c:=\sum_{j=1}^m c_j.
		\]
		If $c=0$, the assertion is immediate. If $c>0$, convexity and
		positive homogeneity give
		\begin{align*}
			F_{s,L}\!\left(\sum_{j=1}^m c_jW_j\right)
			&=
			cF_{s,L}\!\left(
			\sum_{j=1}^m\frac{c_j}{c}W_j
			\right)\\
			&\le
			c\sum_{j=1}^m\frac{c_j}{c}F_{s,L}(W_j)\\
			&=
			\sum_{j=1}^m c_jF_{s,L}(W_j).
		\end{align*}
		This completes the proof.
	\end{proof}
	
	The preceding proposition yields a transported Minkowski inequality.
	
	\begin{lemma}
		\label{lem:transported-Minkowski}
		Let $1/2\le r\le1$, let $B\in\N^{++}$, and let
		$A_1,\ldots,A_m,E\in\N_+$. For $c_1,\ldots,c_m\ge0$, define
		\begin{equation}\label{eq:transported-sum}
			D:=B^{-1/2}
			\left(\sum_{j=1}^m c_j(B^{1/2}A_jB^{1/2})^r\right)^{1/r}
			B^{-1/2}.
		\end{equation}
		Then
		\begin{equation}\label{eq:transported-Minkowski}
			\K_r(D,E)\le\sum_{j=1}^m c_j\K_r(A_j,E).
		\end{equation}
	\end{lemma}
	
	\begin{proof}
		Put $s=1/r$,
		\[
		W_j=(B^{1/2}A_jB^{1/2})^r,
		\qquad
		L=B^{-1/2}E^{1/2}.
		\]
		Then
		\[
		L^*\left(\sum_{j=1}^mc_jW_j\right)^sL =E^{1/2}DE^{1/2},
		\qquad
		L^*W_j^sL=E^{1/2}A_jE^{1/2}.
		\]
 Since $1/s=r$, \eqref{eq:CL-subadditivity} and
		\eqref{eq:kernel-basic} yield
		\begin{align*}
			\K_r(D,E)
			&=
			\K_r(E,D)\\
			&=
			F_{s,L}\!\left(\sum_{j=1}^m c_jW_j\right)\\
			&\leq
			\sum_{j=1}^mc_jF_{s,L}(W_j)\\
			&=
			\sum_{j=1}^mc_j\K_r(E,A_j)
			=
			\sum_{j=1}^mc_j\K_r(A_j,E).
	\end{align*} 
	\end{proof}
	
	Its key consequence is the following one-versus-many estimate.
	
	\begin{proposition}
		\label{prop:one-vs-many}
		Let $1/2\le r\le1$, $A_1,\ldots,A_m,R\in\N_+$, and
		$c_1,\ldots,c_m\ge0$. Then
		\begin{equation}\label{eq:one-vs-many}
			\left(\sum_{i=1}^m c_i\K_r(A_i,R)\right)^2
			\le
			\K_r(R,R)
			\sum_{i,j=1}^m c_ic_j\K_r(A_i,A_j).
		\end{equation}
	\end{proposition}
	
	\begin{proof}
		 First assume $R\in\N^{++}$, and set
			\[
			D:=
			R^{-1/2}
			\left(
			\sum_{i=1}^m c_i(R^{1/2}A_iR^{1/2})^r
			\right)^{1/r}
			R^{-1/2}.
			\] 
			Then, by \eqref{eq:kernel} and the symmetry of $\K_r$,
			\begin{equation}\label{eq:exact-at-R}
				\K_r(D,R)
				=
				\K_r(R,D)
				=
				\sum_{i=1}^m c_i\K_r(A_i,R).
		\end{equation}
 
		By \eqref{eq:pairwise-CS},
		\[
		\K_r(D,R)^2\le\K_r(D,D)\K_r(R,R).
		\]
Applying Lemma~\ref{lem:transported-Minkowski} first with
$E=D$, and then, for each $i$, with $E=A_i$, and using the
symmetry of $\K_r$, gives 
  \[
\K_r(D,D)
\leq
\sum_{i=1}^m c_i\K_r(D,A_i)
\leq
\sum_{i,j=1}^mc_ic_j\K_r(A_i,A_j).
\] 
Together with \eqref{eq:exact-at-R}, this proves
\eqref{eq:one-vs-many}.  For singular $R$, replace it by
		$R+\varepsilon\one_\N$ and let $\varepsilon\downarrow0$. This
		regularization remains inside the present finite algebra, and norm
		continuity of functional calculus passes every term to the limit.
	\end{proof}
	
	\begin{proof}[Proof of Theorem~\ref{thm:three-point}]  
		The matrix in \eqref{eq:three-point} is real and symmetric, so it suffices
		to test its quadratic form on real vectors.    If $x=(x_1,x_2,x_3)$ has either all nonnegative or all
			nonpositive coordinates, then $x_ix_j\geq0$ for every $i,j$.
			Since every entry of $\K_r$ is nonnegative, each term in
			\[
			x^{\mathsf T}\K_r x
			=
			\sum_{i,j=1}^3 x_ix_j(\K_r)_{ij}
			\]
			is nonnegative. Hence $x^{\mathsf T}\K_r x\geq0$.
		
		After a permutation of the indices and, if necessary, multiplication of
		the vector by $-1$, every remaining sign pattern can be written as
		$(u,v,-w)$, where $u,v,w\ge0$. Set
		\[
		S:=
		u^2\K_r(P_1,P_1)
		+v^2\K_r(P_2,P_2)
		+2uv\K_r(P_1,P_2)
		\]
		and
		\[
		L:=
		u\K_r(P_1,P_3)
		+v\K_r(P_2,P_3).
		\]
		Since $u,v\ge0$ and $\K_r$ is entrywise nonnegative, one has
		\[
		S\ge0,
		\qquad
		L\ge0.
		\]
		
		Apply Proposition~\ref{prop:one-vs-many} with
		\[
		m=2,\qquad
		(A_1,A_2,R)=(P_1,P_2,P_3),
		\qquad
		(c_1,c_2)=(u,v).
		\]
		The resulting inequality is
		\[
		\left(
		u\K_r(P_1,P_3)+v\K_r(P_2,P_3)
		\right)^2
		\le
		\K_r(P_3,P_3)
		\sum_{i,j=1}^2c_ic_j\K_r(A_i,A_j).
		\]
		By the symmetry of $\K_r$,
		\begin{align*}
			\sum_{i,j=1}^2c_ic_j\K_r(A_i,A_j)
			&=
			u^2\K_r(P_1,P_1)
			+v^2\K_r(P_2,P_2)\\
			&\qquad
			+2uv\K_r(P_1,P_2)
			=S.
		\end{align*}
		Therefore
		\[
		L^2\le S\K_r(P_3,P_3).
		\]
		Since $L\ge0$, it follows that
		\[
		L\le\sqrt{S\K_r(P_3,P_3)}.
		\]
		
		Consequently,
		\begin{align*}
			&(u,v,-w)
			\bigl[\K_r(P_i,P_j)\bigr]_{i,j=1}^3
			(u,v,-w)^T\\
			&\qquad
			=S-2wL+w^2\K_r(P_3,P_3)\\
			&\qquad
			\ge
			S-2w\sqrt{S\K_r(P_3,P_3)}
			+w^2\K_r(P_3,P_3)\\
			&\qquad
			=
			\left(
			\sqrt S-w\sqrt{\K_r(P_3,P_3)}
			\right)^2
			\ge0.
		\end{align*}
		Thus the quadratic form is nonnegative on every real vector, and hence
		$\bigl[\K_r(P_i,P_j)\bigr]_{i,j=1}^3$ is positive semidefinite.
	\end{proof}
	
	We now make precise the sharp order-three scope of the theorem.
	
\begin{remark}\label{rem:three-point-scope}
	The order-three conclusion is sharp for every $1/2\le r<1$. Let
	\[
	u_k=
	\begin{pmatrix}
		\cos(k\pi/4)\\
		\sin(k\pi/4)
	\end{pmatrix},
	\qquad
	P_k=u_ku_k^*,
	\qquad k=0,1,2,3.
	\]
	These are rank-one density matrices in $M_2(\C)$. Since
	$P_i^{1/2}=P_i$ and
	\[
	P_iP_jP_i=|\langle u_i,u_j\rangle|^2P_i,
	\]
	we have, for every $r>0$,
	\[
	\K_r^{\Tr}(P_i,P_j)=|\langle u_i,u_j\rangle|^{2r}.
	\]
	Put $a_r:=2^{-r}$. Then
	\[
	\bigl[\K_r^{\Tr}(P_i,P_j)\bigr]_{i,j=0}^3
	=
	\begin{pmatrix}
		1&a_r&0&a_r\\
		a_r&1&a_r&0\\
		0&a_r&1&a_r\\
		a_r&0&a_r&1
	\end{pmatrix}.
	\]
	The vector $(1,-1,1,-1)^T$ is an eigenvector with eigenvalue
	\[
	1-2a_r=1-2^{1-r},
	\]
	which is negative whenever $1/2\le r<1$. Hence the four-point kernel
	matrix is not positive semidefinite for any $1/2\le r<1$, even on pure
	qubit states, and the order-three conclusion of
	Theorem~\ref{thm:three-point} is sharp throughout this range. At $r=1$,
	by contrast, $\K_1^{\Tr}(P,Q)=\Tr(PQ)$ is a Hilbert--Schmidt Gram kernel
	whose kernel matrices are positive semidefinite at every order.
	
	This also explains the role of the one-versus-two estimate: every
	mixed-sign vector in $\mathbb R^3$ has such a sign pattern, whereas four
	points already permit a two-versus-two pattern. Related higher-block
	compression obstructions appear in \cite{Audenaert,LinDriessche}.
\end{remark}
	
	The three-point theorem also gives a Euclidean realization of $d_t$.
	
	\begin{corollary}
		\label{cor:finite-realization}
		Let $1\le t\le2$ and $A_1,A_2,A_3\in\N_+$. Then there exist vectors
		$\xi_1,\xi_2,\xi_3\in\mathbb R^3$, depending on the given triple, such
		that
		\begin{equation}\label{eq:finite-realization}
			\|\xi_i-\xi_j\|_{\mathbb R^3}
			=\sqrt2\,d_t(A_i,A_j),
			\qquad 1\le i,j\le3.
		\end{equation}
		In particular, $d_t$ satisfies the triangle inequality on $\N_+$.
		Moreover, for all $A,B\in\N_+$,
		\begin{equation}\label{eq:endpoint-order-finite}
			Q_1(A,B)\le Q_t(A,B)\le Q_2(A,B),
			\qquad
			d_2(A,B)\le d_t(A,B)\le d_1(A,B).
		\end{equation}
	\end{corollary}
	
	\begin{proof}
		Put
		\[
		r:=\frac1t,
		\qquad
		P_i:=A_i^{t/2},
		\quad i=1,2,3.
		\]
		Since $1\le t\le2$, one has $1/2\le r\le1$. By
		Theorem~\ref{thm:three-point}, the real symmetric matrix
		\[
		G:=
		\bigl[\K_r(P_i,P_j)\bigr]_{i,j=1}^3
		=
		\bigl[
		\K_{1/t}(A_i^{t/2},A_j^{t/2})
		\bigr]_{i,j=1}^3
		=
		\bigl[Q_t(A_i,A_j)\bigr]_{i,j=1}^3
		\]
		is positive semidefinite. Its diagonal entries satisfy
		\[
		G_{ii}
		=
		Q_t(A_i,A_i)
		=
		\nu(A_i),
		\qquad i=1,2,3.
		\]
		
		Let $e_1,e_2,e_3$ be the standard basis of $\mathbb R^3$, and define
		\[
		\xi_i:=G^{1/2}e_i,
		\qquad i=1,2,3.
		\]
		Since $G^{1/2}$ is real symmetric,
		\[
		\langle\xi_i,\xi_j\rangle_{\mathbb R^3}
		=
		e_i^TG^{1/2}G^{1/2}e_j
		=
		e_i^TGe_j
		=
		G_{ij}
		=
		Q_t(A_i,A_j).
		\]
		Consequently,
		\begin{align*}
			\|\xi_i-\xi_j\|_{\mathbb R^3}^2
			&=
			\langle\xi_i,\xi_i\rangle_{\mathbb R^3}
			+\langle\xi_j,\xi_j\rangle_{\mathbb R^3}
			-2\langle\xi_i,\xi_j\rangle_{\mathbb R^3}\\
			&=
			\nu(A_i)+\nu(A_j)-2Q_t(A_i,A_j)\\
			&=
			2d_t(A_i,A_j)^2.
		\end{align*}
		Taking square roots proves \eqref{eq:finite-realization}, and the
		triangle inequality for $d_t$ follows from the Euclidean triangle
		inequality.
		
		It remains to prove the endpoint ordering. The Araki--Lieb--Thirring
		inequality \cite{Araki}, in Kosaki's von Neumann algebra form
		\cite[Corollary~3]{KosakiALT}, gives
		\[
		\nu\!\left(|ab|^{\rho q}\right)
		\le
		\nu\!\left(|a^\rho b^\rho|^q\right),
		\qquad \rho\ge1,\quad q>0.
		\]
		For $X,Y\in\N^{++}$, apply this with
		\[
		a=Y^{1/2},\qquad b=X^{1/2},\qquad q=2p.
		\]
		Since
		\[
		|ab|^2=X^{1/2}YX^{1/2},
		\qquad
		|a^\rho b^\rho|^2
		=X^{\rho/2}Y^\rho X^{\rho/2},
		\]
		we obtain
		\begin{equation}\label{eq:finite-tracial-ALT}
			\nu\!\left(
			\left(X^{1/2}YX^{1/2}\right)^{\rho p}
			\right)
			\le
			\nu\!\left(
			\left(X^{\rho/2}Y^\rho X^{\rho/2}\right)^p
			\right).
		\end{equation}
		For arbitrary $X,Y\in\N_+$, apply the invertible case to
		$X+\varepsilon\one$ and $Y+\varepsilon\one$, and let
		$\varepsilon\downarrow0$. Every power and product in
		\eqref{eq:finite-tracial-ALT} then converges in norm, and $\nu$ is
		bounded because it is finite. Thus \eqref{eq:finite-tracial-ALT} also
		holds for singular $X$ and $Y$.
		
		Now let $1\le t_1\le t_2\le2$ and set
		\[
		\rho=\frac{t_2}{t_1},
		\qquad
		p=\frac1{t_2},
		\qquad
		X=A^{t_1/2},
		\qquad
		Y=B^{t_1/2}.
		\]
		Then
		\[
		X^{1/2}=A^{t_1/4},
		\qquad
		\rho p=\frac1{t_1},
		\qquad
		X^{\rho/2}=A^{t_2/4},
		\qquad
		Y^\rho=B^{t_2/2}.
		\]
		Thus the two sides of \eqref{eq:finite-tracial-ALT} are respectively
		$Q_{t_1}(A,B)$ and $Q_{t_2}(A,B)$, and hence
		\[
		Q_{t_1}(A,B)\le Q_{t_2}(A,B).
		\]
		Hence $t\mapsto Q_t(A,B)$ is increasing on $[1,2]$, and consequently
		\[
		Q_1(A,B)\le Q_t(A,B)\le Q_2(A,B).
		\]
		Subtracting these inequalities from the common arithmetic term
		$\bigl(\nu(A)+\nu(B)\bigr)/2$ and taking square roots gives
		\[
		d_2(A,B)\le d_t(A,B)\le d_1(A,B).
		\]
	\end{proof}
	\section{Three-block Schatten compression and the
		Lin--van den Driessche conjecture}
	\label{sec:block}
	
	We first prove a three-block compression theorem for tracial von Neumann
	algebras. Its finite-dimensional specialization resolves
	\cite[Conjecture~5.1]{LinDriessche}.
	
	 \begin{proof}[Proof of Theorem~\ref{thm:tracial-block-compression}]
We equip $M_3(\mathcal R)$ with the trace
$\widetilde\omega$ defined, for
$A=[A_{kl}]_{k,l=1}^3\in M_3(\mathcal R)$, by
\[
\widetilde\omega(A)
:=
\sum_{k=1}^3\omega(A_{kk}).
\]
This is a faithful normal finite trace on $M_3(\mathcal R)$.
			
			Since $H\geq0$, write $H=X^*X$. Let
			$X_1,X_2,X_3\in M_{3,1}(\mathcal R)$ be the block columns of $X$,
			and set
			\[
			P_i:=X_iX_i^*\in M_3(\mathcal R)_+,
			\qquad i=1,2,3.
			\]
			Then
			\[
			H_{ij}=X_i^*X_j.
			\]
			
			Fix $i,j$, and put
			\[
			Z_{ij}:=P_i^{1/2}X_j\in M_{3,1}(\mathcal R).
			\] 
			Using $P_i=X_iX_i^*$ and $P_j=X_jX_j^*$, we obtain
			\begin{align*}
				Z_{ij}^*Z_{ij}
				&=
				X_j^*P_iX_j
				=
				(X_i^*X_j)^*(X_i^*X_j)
				=
				|H_{ij}|^2,\\
				Z_{ij}Z_{ij}^*
				&=
				P_i^{1/2}P_jP_i^{1/2}.
			\end{align*} 
			
			To justify the rectangular trace identity, take the polar
			decomposition
			\[
			Z_{ij}=V|Z_{ij}|,
			\qquad V\in M_{3,1}(\mathcal R),
			\]
			obtained by viewing $Z_{ij}$ as an off-diagonal corner of
			$M_4(\mathcal R)$. Let $e=s(|Z_{ij}|)$ be the support projection
			of $|Z_{ij}|$. Then
			\[
			V^*V=e,
			\qquad
			|Z_{ij}|=e|Z_{ij}|e.
			\]
			For every $r>0$, functional calculus on the support gives
			\[
			(Z_{ij}Z_{ij}^*)^r
			=
			V|Z_{ij}|^{2r}V^*
			=
			V(Z_{ij}^*Z_{ij})^rV^*.
			\]
			Writing $V=(v_1,v_2,v_3)^{\mathsf T}$, the definition of
			$\widetilde\omega$ and the trace property of $\omega$ yield
			\begin{align*}
				\widetilde\omega((Z_{ij}Z_{ij}^*)^r)
				&=
				\sum_{k=1}^3
				\omega\!\left(v_k(Z_{ij}^*Z_{ij})^rv_k^*\right)\\
				&=
				\omega\!\left(
				(Z_{ij}^*Z_{ij})^r\sum_{k=1}^3v_k^*v_k
				\right)\\
				&=
				\omega\!\left((Z_{ij}^*Z_{ij})^r e\right)\\
				&=
				\omega((Z_{ij}^*Z_{ij})^r).
			\end{align*}
			Here the last equality follows from
			$s((Z_{ij}^*Z_{ij})^r)=e$, since $r>0$.
			Thus the identity also holds when $Z_{ij}$ has a nontrivial kernel.
			
			Now set $r=p/2\in[1/2,1]$. It follows that
			\begin{align*}
				\omega(|H_{ij}|^p)
				&=
				\omega((Z_{ij}^*Z_{ij})^r)\\
				&=
				\widetilde\omega((Z_{ij}Z_{ij}^*)^r)\\
				&=
				\widetilde\omega\!\left(
				(P_i^{1/2}P_jP_i^{1/2})^r
				\right)\\
				&=
				\K_r^{\widetilde\omega}(P_i,P_j).
			\end{align*}

			Consequently,
			\[
			\bigl[\omega(|H_{ij}|^p)\bigr]_{i,j=1}^3
			=
			\bigl[
			\K_r^{\widetilde\omega}(P_i,P_j)
			\bigr]_{i,j=1}^3.
			\]
			The matrix on the right is positive semidefinite by
			Theorem~\ref{thm:three-point}, applied to the finite tracial
			von Neumann algebra
			$(M_3(\mathcal R),\widetilde\omega)$.
	\end{proof}
	
	The same factorization, allowing rectangular block columns, proves the
	finite-dimensional conjecture of Lin and van den Driessche.
	
	\begin{corollary}
		\label{cor:block-conjecture}
		Let $H=[H_{ij}]_{i,j=1}^3\ge0$ be a positive semidefinite block matrix
		acting on
		\[
		\C^{n_1}\oplus\C^{n_2}\oplus\C^{n_3}.
		\]
		Then, for every $1\le p\le2$,
		\begin{equation}\label{eq:matrix-block-compression}
			\bigl[\|H_{ij}\|_p^p\bigr]_{i,j=1}^3\ge0.
		\end{equation}
	\end{corollary}
	
	\begin{proof}
		Set $N:=n_1+n_2+n_3$ and let $X:=H^{1/2}$, regarded as an
		operator on
		\[
		\C^N=\C^{n_1}\oplus\C^{n_2}\oplus\C^{n_3}.
		\]
		Write
		\[
		X=\begin{pmatrix}X_1&X_2&X_3\end{pmatrix},
		\qquad
		X_i:\C^{n_i}\longrightarrow\C^N,
		\]
		for its block columns. Since $H=X^*X$, we have
		\[
		H_{ij}=X_i^*X_j.
		\]
		Put $P_i:=X_iX_i^*\in M_N(\C)_+$. The same rectangular
		trace calculation as above gives
		\[
		\|H_{ij}\|_p^p
		=
		\Tr\!\left(
		(P_i^{1/2}P_jP_i^{1/2})^{p/2}
		\right)
		=
		\K_{p/2}^{\Tr}(P_i,P_j).
		\]
		Hence
		\[
		\bigl[\|H_{ij}\|_p^p\bigr]_{i,j=1}^3
		=
		\bigl[
		\K_{p/2}^{\Tr}(P_i,P_j)
		\bigr]_{i,j=1}^3,
		\]
		which is positive semidefinite by
		Theorem~\ref{thm:three-point}.
	\end{proof}
	
	We record the endpoint interpretation of this corollary.
	
	\begin{remark}
		At $p=1$, Corollary~\ref{cor:block-conjecture} recovers Drury's theorem
		\cite{Drury}. At $p=2$, the entries
		$\|H_{ij}\|_2^2=\Tr(P_iP_j)$ form a Hilbert--Schmidt Gram matrix for any
		number of blocks. The corollary fills the interval between these endpoints in the
		genuinely three-block range.
	\end{remark}
	
	\section{Semifinite metricity, topology, and completeness}
	\label{sec:semifinite}
	
Having established the finite-tracial kernel and block-compression
	results, we return to the general semifinite setting of the
	Introduction. Thus, throughout this section, $(\M,\tau)$ denotes an
	arbitrary von Neumann algebra equipped with a faithful normal
	semifinite trace. For bounded positive elements in a unital tracial
	$C^*$-algebra, Zhang proved that
	\[
	Q_t(A,B)\leq\sqrt{\tau(A)\tau(B)},
	\qquad t>0;
	\]
	see \cite[Theorem~3.3]{Zhang}. Since elements of
	$L^1(\M,\tau)_+$ may be unbounded and $\tau(\one)$ may be infinite,
	that result does not directly cover the present setting. The following
	lemma supplies the required extension to the full semifinite
	$L^1$-cone and records the symmetry and continuity properties needed
	below.
	
	\begin{lemma}
		\label{lem:Qt-analytic}
		For every $t>0$ and $A,B\in L^1(\M,\tau)_+$, the quantity
		$Q_t(A,B)$ in \eqref{eq:Qt-intro} is finite and satisfies
		\begin{equation}\label{eq:Qt-basic-semifinite}
			Q_t(A,B)=Q_t(B,A),
			\qquad
			0\le Q_t(A,B)\le\sqrt{\tau(A)\tau(B)}.
		\end{equation}
		For $1\le t\le2$, the map
		\[
		(A,B)\longmapsto Q_t(A,B)
		\]
		is jointly continuous in the $L^1$-topology.
	\end{lemma}
	
	\begin{proof}
		Put
		\[
		p:=\frac2t,
		\qquad
		q:=\frac4t.
		\]
		Since $A,B\in L^1(\M,\tau)_+$, functional calculus gives
		\[
		A^{t/4},B^{t/4}\in L^q(\M,\tau).
		\]
		The generalized noncommutative H\"older inequality yields
		\begin{align*}
			\|B^{t/4}A^{t/4}\|_p
			&\le
			\|B^{t/4}\|_q\|A^{t/4}\|_q\\
			&=
			\tau(B)^{t/4}\tau(A)^{t/4}.
		\end{align*}
		Raising this inequality to the power $p=2/t$ gives
		\[
		\|B^{t/4}A^{t/4}\|_p^p
		\le
		\sqrt{\tau(A)\tau(B)},
		\]
		which proves finiteness and the upper bound in
		\eqref{eq:Qt-basic-semifinite}.
		
		Moreover,
		\[
		\bigl|B^{t/4}A^{t/4}\bigr|^2
		=
		A^{t/4}B^{t/2}A^{t/4}.
		\]
		Therefore
		\[
		\|B^{t/4}A^{t/4}\|_p^p
		=
		\tau\!\left(
		\left(A^{t/4}B^{t/2}A^{t/4}\right)^{1/t}
		\right),
		\]
		which proves the equality of the two expressions defining $Q_t$.
		Since the $L^p$-(quasi-)norm is invariant under adjoints,
		\[
		\|B^{t/4}A^{t/4}\|_p^p
		=
		\|A^{t/4}B^{t/4}\|_p^p.
		\]
		Thus $Q_t(A,B)=Q_t(B,A)$.
		For the continuity assertion, fix $1\le t\le2$. Then $p=2/t\ge1$. Let
		\[
		A_n\longrightarrow A,
		\qquad
		B_n\longrightarrow B
		\quad\text{in }L^1(\M,\tau),
		\]
		where $A_n,A,B_n,B\in L^1(\M,\tau)_+$. Applying the positive-power
		estimate \cite[Lemma~2.1]{Ricard} with
		\[
		\theta=\frac t4,
		\qquad
		q=\frac4t,
		\]
		so that $\theta q=1$, we obtain
		\[
		\|A_n^{t/4}-A^{t/4}\|_q
		\le
		\|A_n-A\|_1^{t/4}
		\longrightarrow0
		\]
		and, similarly,
		\[
		\|B_n^{t/4}-B^{t/4}\|_q
		\le
		\|B_n-B\|_1^{t/4}
		\longrightarrow0.
		\]
		In particular, the sequence $\|A_n^{t/4}\|_q$ is bounded.
		Since $1/p=1/q+1/q$, H\"older's inequality gives
		\begin{align*}
			&\|B_n^{t/4}A_n^{t/4}
			-B^{t/4}A^{t/4}\|_p\\
			&\quad\le
			\|B_n^{t/4}-B^{t/4}\|_q
			\,\|A_n^{t/4}\|_q\\
			&\qquad+
			\|B^{t/4}\|_q
			\,\|A_n^{t/4}-A^{t/4}\|_q
			\longrightarrow0.
		\end{align*}
		Since $p\ge1$, the map
		\[
		X\longmapsto\|X\|_p^p
		\]
		is continuous on $L^p(\M,\tau)$. Consequently,
		\[
		Q_t(A_n,B_n)\longrightarrow Q_t(A,B),
		\]
		which proves joint $L^1$-continuity.
	\end{proof}
	\begin{proof}[Proof of Theorem~\ref{thm:classification}\textup{(i)}]
		Fix $1\le t\le2$.
		We shall also prove the following quantitative estimates. For
		$A,B\in L^1(\M,\tau)_+$, set
		\[
		a:=\tau(A),
		\qquad
		b:=\tau(B).
		\]
		Then
		\begin{equation}\label{eq:endpoint-bounds}
			\frac12(\sqrt a-\sqrt b)^2
			\le d_2(A,B)^2
			\le d_t(A,B)^2
			\le d_1(A,B)^2,
		\end{equation}
		\begin{equation}\label{eq:Powers-Stormer-bound}
			d_1(A,B)^2
			=
			\frac12\|A^{1/2}-B^{1/2}\|_2^2
			\le
			\frac12\|A-B\|_1,
		\end{equation}
		and
		\begin{equation}\label{eq:FvdG-bound}
			\|A-B\|_1^2
			\le
			(a+b)^2-4Q_2^\tau(A,B)^2
			\le
			4(a+b)d_2(A,B)^2.
		\end{equation}
		
		By Lemma~\ref{lem:Qt-analytic} and the arithmetic--geometric mean
		inequality,
		\[
		Q_t(A,B)
		\le \sqrt{\tau(A)\tau(B)}
		\le \frac{\tau(A)+\tau(B)}2.
		\]
		Thus $d_t(A,B)^2\ge0$. Symmetry follows from the symmetry of $Q_t$,
		and
		\[
		Q_t(A,A)=\|A^{t/2}\|_{2/t}^{2/t}=\tau(A)
		\]
		gives $d_t(A,A)=0$.
		
		We next extend the finite three-point positivity to the semifinite
		$L^1$-cone. Fix $A_1,A_2,A_3\in L^1(\M,\tau)_+$, put
		\[
		S:=A_1+A_2+A_3,
		\qquad
		e_n:=\one_{[1/n,n]}(S),
		\]
		and set
		\[
		A_{i,n}:=e_nA_ie_n,\qquad i=1,2,3.
		\]
		Let $s(S):=\one_{(0,\infty)}(S)$ denote the support projection of
		$S$. Since the intervals $[1/n,n]$ increase to $(0,\infty)$, we
		have $e_n\uparrow s(S)$. On the range of $e_n$, the spectrum of
		$S$ lies in $[1/n,n]$; hence
		\[
		n^{-1}e_n
		\leq
		e_nSe_n
		\leq
		ne_n.
		\]
		Consequently,
		\[
		\tau(e_n)
		\leq
		n\tau(e_nSe_n)
		\leq
		n\tau(S)
		=
		n\sum_{i=1}^3\tau(A_i)
		<\infty.
		\]
		Since $0\leq A_i\leq S$, we also have
		\[
		0\leq A_{i,n}
		\leq e_nSe_n
		\leq ne_n.
		\]
		Thus $A_{1,n},A_{2,n},A_{3,n}$ are bounded positive elements of the
		finite-tracial corner $e_n\M e_n$, equipped with the restricted
		trace.
		
		We also have $A_{i,n}\to A_i$ in $L^1(\M,\tau)$. Indeed, let
		$0\le X\le S$ and put $X_n=e_nXe_n$. Since $s(X)\le s(S)$ and
		$e_n\uparrow s(S)$, the normality of $\tau$ gives
		\begin{align*}
			\delta_n(X)^2
			&:=\|(\one-e_n)X^{1/2}\|_2^2\\
			&=\tau\!\left(X^{1/2}(\one-e_n)X^{1/2}\right)
			\longrightarrow
			\tau\!\left(X^{1/2}(\one-s(S))X^{1/2}\right)=0.
		\end{align*}
		Using
		\[
		X-e_nXe_n=(\one-e_n)X+e_nX(\one-e_n)
		\]
		and noncommutative H\"older, we obtain
		\begin{align*}
			\|X-e_nXe_n\|_1
			&\le
			\|(\one-e_n)X^{1/2}\|_2\|X^{1/2}\|_2\\
			&\qquad+
			\|e_nX^{1/2}\|_2\|X^{1/2}(\one-e_n)\|_2\\
			&\le
			2\|X^{1/2}\|_2\,\delta_n(X)\longrightarrow0,
		\end{align*}
		where
		$\|X^{1/2}(\one-e_n)\|_2
		=\|(\one-e_n)X^{1/2}\|_2$ by traciality.
		Since $0\leq A_i\leq S$ and $A_{i,n}=e_nA_ie_n$, the
		preceding estimate applied with $X=A_i$ gives
		\[
		\|A_i-A_{i,n}\|_1\longrightarrow0,
		\qquad i=1,2,3.
		\]

			Applying Theorem~\ref{thm:three-point} in $e_n\M e_n$ with
			the restricted trace, $r=1/t$, and $P_i=A_{i,n}^{t/2}$,
			we obtain
		
		\[
		\bigl[Q_t(A_{i,n},A_{j,n})\bigr]_{i,j=1}^3\ge0.
		\]
		By the joint $L^1$-continuity from Lemma~\ref{lem:Qt-analytic},
		\[
		Q_t(A_{i,n},A_{j,n})
		\longrightarrow
		Q_t(A_i,A_j).
		\]
		Since the cone of positive semidefinite $3\times3$ matrices is closed,
		we obtain
		\[
		\bigl[Q_t(A_i,A_j)\bigr]_{i,j=1}^3\ge0.
		\]
		
		Now take $A_1=A$, $A_2=B$, and $A_3=C$. The preceding positivity
		gives vectors $\xi_1,\xi_2,\xi_3$ in a Hilbert space such that
		\[
		Q_t(A_i,A_j)=\langle \xi_i,\xi_j\rangle.
		\]
		Since $Q_t(A_i,A_i)=\tau(A_i)$, we have
		\[
		d_t(A_i,A_j)^2
		=
		\frac12\|\xi_i-\xi_j\|^2.
		\]
		Therefore
		\[
		d_t(A,C)
		=
		\frac1{\sqrt2}\|\xi_1-\xi_3\|
		\le
		\frac1{\sqrt2}\|\xi_1-\xi_2\|
		+
		\frac1{\sqrt2}\|\xi_2-\xi_3\|,
		\]
		that is,
		\[
		d_t(A,C)\le d_t(A,B)+d_t(B,C).
		\]
		This proves the triangle inequality.

			For the endpoint comparison, repeat the preceding approximation
			argument with
			\[
			S:=A+B,\qquad e_n:=\one_{[1/n,n]}(S),
			\]
			and set
			\[
			A_n:=e_nAe_n,\qquad B_n:=e_nBe_n.
			\]
			As above,
			\[
			\tau(e_n)\le n\tau(S)<\infty,\qquad
			0\le A_n\le ne_n,\qquad 0\le B_n\le ne_n.
			\]
			Thus $A_n,B_n$ are bounded positive elements of the
			finite-tracial corner $e_n\M e_n$, equipped with the
			restricted trace. Moreover,
			\[
			\|A_n-A\|_1+\|B_n-B\|_1\longrightarrow0.
			\]
			Applying Corollary~\ref{cor:finite-realization} in this corner gives
			\[
			d_2(A_n,B_n)\le d_t(A_n,B_n)\le d_1(A_n,B_n).
			\]
			By the continuity of the trace on $L^1$,
			Lemma~\ref{lem:Qt-analytic}, and the definition of $d_s$,
			\[
			d_s(A_n,B_n)\longrightarrow d_s(A,B),
			\qquad s\in\{1,t,2\}.
			\]
			Passing to the limit yields
			\[
			d_2(A,B)\le d_t(A,B)\le d_1(A,B).
			\]
		
		These are the middle inequalities in \eqref{eq:endpoint-bounds}.
		The first inequality in \eqref{eq:endpoint-bounds} follows from
		$Q_2^\tau(A,B)\le\sqrt{ab}$.
		
		For $t=1$, traciality gives
		\[
		Q_1(A,B)=\tau(A^{1/2}B^{1/2}),
		\]
		and hence
		\[
		d_1(A,B)^2
		=
		\frac12\|A^{1/2}-B^{1/2}\|_2^2.
		\]
		The noncommutative $L^p$-extension of the Powers--St\o rmer
		inequality, applied with $p=1$ and $\theta=1/2$, gives
		\[
		\|A^{1/2}-B^{1/2}\|_2
		\leq
		\|A-B\|_1^{1/2};
		\]
		see \cite[Appendix, Theorem, p.~289]{HiaiNakamura}. Squaring this
		inequality and combining it with the preceding identity yields
		\eqref{eq:Powers-Stormer-bound}. The corresponding classical
		bounded-operator inequality appears in
		\cite[Lemma~4.1, p.~21]{PowersStormer}.
		
		It remains to prove \eqref{eq:FvdG-bound}. Let
		\[
		P:=\one_{(0,\infty)}(A-B)
		\]
		and set
		\[
		a_1=\tau(PAP),\quad a_2=\tau((\one-P)A(\one-P)),\quad
		b_1=\tau(PBP),\quad b_2=\tau((\one-P)B(\one-P)).
		\]
		Let
		\[
		X:=A^{1/2}B^{1/2}\in L^1(\M,\tau),
		\]
		and let
		\[
		X=V|X|
		\]
		be its polar decomposition. By the polar decomposition theorem, $V$
		is a partial isometry with initial projection
		\[
		V^*V=s(|X|).
		\]
		Put $U:=V^*$. Then
		\[
		UU^*=V^*V=s(|X|)\leq\one
		\]
		and
		\[
		UX=V^*V|X|=|X|.
		\]
		Hence, by adjoint invariance of the $L^1$-norm and traciality,
		\begin{align*}
			Q_2^\tau(A,B)
			&=\|B^{1/2}A^{1/2}\|_1\\
			&=\|A^{1/2}B^{1/2}\|_1\\
			&=\tau(UX)\\
			&=\tau(XU)\\
			&=\langle A^{1/2},B^{1/2}U\rangle_2\\
			&=\bigl|\langle A^{1/2},B^{1/2}U\rangle_2\bigr|.
		\end{align*}
		We now decompose the two $L^2$-elements according to $P$ and
		$\one-P$:
		\[
		A^{1/2}
		=
		PA^{1/2}+(\one-P)A^{1/2},
		\qquad
		B^{1/2}U
		=
		PB^{1/2}U+(\one-P)B^{1/2}U.
		\]
		Since $P(\one-P)=0$, the corresponding cross terms in the
		$L^2$-inner product vanish. Hence
		\begin{align*}
			\langle A^{1/2},B^{1/2}U\rangle_2
			&=
			\langle PA^{1/2},PB^{1/2}U\rangle_2\\
			&\quad+
			\langle(\one-P)A^{1/2},
			(\one-P)B^{1/2}U\rangle_2.
		\end{align*}
		Taking absolute values and applying the triangle inequality followed by
		the Hilbert-space Cauchy--Schwarz inequality, we obtain
		\begin{align*}
			Q_2^\tau(A,B)
			&\leq
			\|PA^{1/2}\|_2\,\|PB^{1/2}U\|_2\\
			&\quad+
			\|(\one-P)A^{1/2}\|_2\,
			\|(\one-P)B^{1/2}U\|_2.
		\end{align*}
		
		Moreover,
		\[
		\|PA^{1/2}\|_2^2=a_1,
		\qquad
		\|(\one-P)A^{1/2}\|_2^2=a_2,
		\]
		while $UU^*\le\one$ and traciality imply
		\[
		\begin{aligned}
			\|PB^{1/2}U\|_2^2
			=\tau\!\left(U^*B^{1/2}PB^{1/2}U\right)
			\le b_1,\quad
			\|(\one-P)B^{1/2}U\|_2^2
			\le b_2.
		\end{aligned}
		\]
		Consequently,
		\[
		Q_2^\tau(A,B)
		\le
		\sqrt{a_1b_1}+\sqrt{a_2b_2}.
		\]
		Together with
		\[
		\|A-B\|_1=(a_1+b_2)-(b_1+a_2),
		\qquad
		a+b=(a_1+b_2)+(b_1+a_2),
		\]
		and
		\[
		\left(\sqrt{a_1b_1}+\sqrt{a_2b_2}\right)^2
		\le
		(a_1+b_2)(b_1+a_2),
		\]
		this yields
		\[
		\|A-B\|_1^2\le (a+b)^2-4Q_2^\tau(A,B)^2.
		\]
		The second inequality in \eqref{eq:FvdG-bound} follows from
		\begin{align*}
			(a+b)^2-4Q_2^\tau(A,B)^2
			&=(a+b-2Q_2^\tau(A,B))(a+b+2Q_2^\tau(A,B)),\\
			a+b-2Q_2^\tau(A,B)
			&=2d_2(A,B)^2,
		\end{align*}
		and $2Q_2^\tau(A,B)\le a+b$.
		
		The estimates also prove separation. Indeed, if $d_t(A,B)=0$,
		then $d_2(A,B)\leq d_t(A,B)$ implies $d_2(A,B)=0$. Hence
		\eqref{eq:FvdG-bound} gives
		\[
		\|A-B\|_1^2
		\leq
		4\bigl(\tau(A)+\tau(B)\bigr)d_2(A,B)^2
		=0,
		\]
		and therefore $A=B$.
		
		Let $(X_n)$ be $d_t$-Cauchy. By
		\eqref{eq:endpoint-bounds}, the sequence
		$\bigl(\sqrt{\tau(X_n)}\bigr)$ is Cauchy and hence bounded.
		Therefore, $\bigl(\tau(X_n)\bigr)$ is bounded. Since $d_2\le d_t$,
		\eqref{eq:FvdG-bound} implies that $(X_n)$ is $L^1$-Cauchy. Hence
		$X_n\to X\in L^1(\M,\tau)_+$. By
		\eqref{eq:Powers-Stormer-bound},
		\[
		d_1(X_n,X)\to0,
		\]
		and therefore $d_t(X_n,X)\to0$.
		
			Thus $(L^1(\M,\tau)_+,d_t)$ is complete.
		
		Conversely, $L^1$-convergence implies
		$d_t$-convergence by \eqref{eq:Powers-Stormer-bound} and $d_t\le d_1$.
		The same estimates also show that the $d_t$-Cauchy sequences and the
		$L^1$-Cauchy sequences coincide.

			Finally, suppose that $d_t(X_n,X)\to0$ for
			$X_n,X\in L^1(\M,\tau)_+$. By \eqref{eq:endpoint-bounds},
			\[
			\left|\sqrt{\tau(X_n)}-\sqrt{\tau(X)}\right|
			\le \sqrt2\,d_t(X_n,X)\longrightarrow0.
			\]
			Thus $\bigl(\tau(X_n)\bigr)$ is bounded. Using
			\eqref{eq:FvdG-bound} and $d_2\le d_t$, we obtain
			\[
			\|X_n-X\|_1^2
			\le
			4\bigl(\tau(X_n)+\tau(X)\bigr)d_t(X_n,X)^2
			\longrightarrow0.
			\]
			Together with the converse implication proved above, this shows
			that convergence in $d_t$ is equivalent to convergence in $L^1$.
			Since both topologies are metrizable, they coincide.
		
	\end{proof}

 The finite-corner argument above is local to the three operators under
	consideration. Indeed, it uses only the spectral projections
	\[
	e_n=\one_{[1/n,n]}(A_1+A_2+A_3),
	\]
	which increase to the support of the particular sum
	$A_1+A_2+A_3$ and satisfy $\tau(e_n)<\infty$. These projections
	need not increase to $\one$. Consequently, the proof requires neither
	$\tau(\one)<\infty$ nor a global finite-trace exhaustion of $\one$;
	in particular, no $\sigma$-finiteness assumption is implicit.

	\section{Sharp metricity classification}
	\label{sec:classification}
	
We first prove that, for every $0<t<1$, the triangle inequality fails
already on $M_2(\C)^{++}$. For $t>2$, we regularize the rank-one
counterexample of Kom\'alovics and Moln\'ar to obtain the same strict
failure for positive definite $2\times2$ matrices. These
finite-dimensional counterexamples will then be transported through
finite-trace matrix corners. The separately defined logarithmic endpoint
$t=0$ is treated in Section~\ref{sec:zero-endpoint}. Throughout this
section, $Q_t^{\Tr}$ and $d_t^{\Tr}$ denote the quantities associated
with the usual matrix trace.

	\begin{proposition}
		\label{prop:matrix-obstructions}
		Let $t>0$ with $t\notin[1,2]$. There exist
		$A_t,B_t,C_t\in M_2(\C)^{++}$ such that
		\begin{equation}\label{eq:matrix-strict-defect}
			d_t^{\Tr}(A_t,B_t)
			>d_t^{\Tr}(A_t,C_t)+d_t^{\Tr}(C_t,B_t).
		\end{equation}
	\end{proposition}
	
	\begin{proof}
		We treat the two ranges separately.
		
		\smallskip
		\noindent\emph{Case 1: $0<t<1$.}
		Put
		\begin{equation}\label{eq:Clarkson-exponents}
			s:=\frac{2}{2-t},
			\qquad
			q:=\frac2t.
		\end{equation}
		Then $1<s<2$, $q=s/(s-1)>2$, and $1/s+1/q=1$. Let
		\[
		x:=\frac34,
		\qquad
		y:=\frac14,
		\qquad
		S:=x^s+y^s,
		\qquad
		\mu:=S^{q/s},
		\]
		and define
		\begin{equation}\label{eq:Clarkson-alpha-beta}
			\alpha:=\frac{x^s}{S},
			\qquad
			\beta:=\frac{y^s}{S}.
		\end{equation}
		Thus $0<\mu<1$ and $\alpha+\beta=1$. Consider the unit vectors
		\[
		u:=\begin{pmatrix}\sqrt3/2\\[1mm]1/2\end{pmatrix},
		\qquad
		v:=\begin{pmatrix}\sqrt3/2\\[1mm]-1/2\end{pmatrix},
		\]
		the rank-one projections
		\begin{equation}\label{eq:Clarkson-PQ}
			P:=uu^*=\frac14
			\begin{pmatrix}3&\sqrt3\\ \sqrt3&1\end{pmatrix},
			\qquad
			Q:=vv^*=\frac14
			\begin{pmatrix}3&-\sqrt3\\ -\sqrt3&1\end{pmatrix},
		\end{equation}
		and the positive definite matrix
		\begin{equation}\label{eq:Clarkson-C}
			C_t:=\mu\begin{pmatrix}\alpha&0\\0&\beta\end{pmatrix}.
		\end{equation}
		We have $\Tr P=\Tr Q=1$ and $\Tr C_t=\mu$.
		
		Since $\langle u,v\rangle=1/2$, one has $PQP=P/4$. Positive powers
		of a projection equal the projection, and hence
		\[
		Q_t^{\Tr}(P,Q)
		=\Tr\!\left((PQP)^{1/t}\right)=2^{-2/t}=2^{-q}.
		\]
		Consequently,
		\begin{equation}\label{eq:Clarkson-distance-PQ}
			d_t^{\Tr}(P,Q)^2=1-2^{-q}.
		\end{equation}
		
		We next compute the distances to $C_t$. Set
		$r=t/2=1/q$, so that $1-r=1/s$. Since $1+sr=s$,
		\begin{align}
			x\alpha^r+y\beta^r
			&=S^{-r}\bigl(x^{1+sr}+y^{1+sr}\bigr)
			=S^{1-r},\label{eq:Clarkson-expectation}\\
			\mu^r&=S^{1/s}=S^{1-r}.\label{eq:Clarkson-mu-power}
		\end{align}
		It follows that
		\[
		PC_t^{t/2}P
		=\langle u,C_t^ru\rangle P
		=S^{2(1-r)}P
		=\mu^tP.
		\]
		Therefore
		\begin{equation}\label{eq:Clarkson-affinities}
			Q_t^{\Tr}(P,C_t)=\mu.
		\end{equation}
		The squared coordinates of $u$ and $v$ are identical, so the same
		calculation gives $Q_t^{\Tr}(Q,C_t)=\mu$. Hence
		\begin{equation}\label{eq:Clarkson-distances-PC-QC}
			d_t^{\Tr}(P,C_t)^2
			=d_t^{\Tr}(Q,C_t)^2
			=\frac{1-\mu}{2}.
		\end{equation}
		
  For $1<s<2$ and $q=s/(s-1)$, the scalar Clarkson
		inequality \cite{Clarkson} gives, for real $a,b$,
		\begin{equation}\label{eq:scalar-Clarkson}
			|a+b|^q+|a-b|^q
			\leq 2\bigl(|a|^s+|b|^s\bigr)^{q/s}.
		\end{equation}
		We verify strictness for $0<b<a$ using Friedrichs' argument
		as recalled in \cite[pp.~1364--1365]{MaligrandaSabourova}.
		Put $z=b/a\in(0,1)$ and, for $z^{s-1}\leq\lambda\leq1$, set
		\[
		\begin{aligned}
			F(\lambda)
			&=(1+\lambda^{1-q}z)(1+\lambda z)^{q-1}\\
			&\quad +(1-\lambda^{1-q}z)(1-\lambda z)^{q-1}.
		\end{aligned}
		\]
		Differentiation gives
		\[
		F'(\lambda)
		=(q-1)z(1-\lambda^{-q})
		\bigl((1+\lambda z)^{q-2}-(1-\lambda z)^{q-2}\bigr)<0
		\]
		for $z^{s-1}<\lambda<1$, since $q>2$.
		Using $(s-1)(q-1)=1$ and $q-1=q/s$, we obtain
		\[
		(1+z)^q+(1-z)^q
		=F(1)<F(z^{s-1})=2(1+z^s)^{q/s}.
		\]
		Multiplication by $a^q$ proves the required strictness.
		By symmetry and direct substitution, equality in
		\eqref{eq:scalar-Clarkson} for $a,b\geq0$ holds exactly when
		$ab=0$ or $a=b$. In particular, $a=x=3/4$ and $b=y=1/4$
		give
		\begin{equation}\label{eq:strict-Clarkson}
			1+2^{-q}<2(x^s+y^s)^{q/s}=2\mu.
	\end{equation}
		Combining \eqref{eq:Clarkson-distance-PQ},
		\eqref{eq:Clarkson-distances-PC-QC}, and \eqref{eq:strict-Clarkson}, we get
		\[
		d_t^{\Tr}(P,Q)^2-4d_t^{\Tr}(P,C_t)^2
		=2\mu-1-2^{-q}>0.
		\]
		Thus
		\begin{equation}\label{eq:strict-defect-below}
			d_t^{\Tr}(P,Q)
			>d_t^{\Tr}(P,C_t)+d_t^{\Tr}(C_t,Q).
		\end{equation}
		
		Although $P$ and $Q$ have rank one, the violation can be made positive
		definite. For $\varepsilon>0$, set
		\[
		P_\varepsilon=P+\varepsilon I_2,
		\qquad
		Q_\varepsilon=Q+\varepsilon I_2,
		\qquad
		C_{t,\varepsilon}=C_t+\varepsilon I_2.
		\]
		For fixed $t>0$, continuous functional calculus makes
		$(A,B)\mapsto d_t^{\Tr}(A,B)$ continuous on
		$M_2(\C)_+\times M_2(\C)_+$. Hence the strict defect in
		\eqref{eq:strict-defect-below} persists for all sufficiently small
		$\varepsilon>0$.
		
		\smallskip
		\noindent\emph{Case 2: $t>2$.}
		For rank-one projections $P_u=uu^*$ and $P_v=vv^*$, one has
		\[
		d_t^{\Tr}(P_u,P_v)^2
		=1-|\langle u,v\rangle|^{2/t}.
		\]
		Kom\'alovics and Moln\'ar proved that, when $t>2$, the restriction of
		$d_t^{\Tr}$ to rank-one projections already violates the triangle
		inequality; see \cite[Proposition~8]{KomalovicsMolnar}. Thus there are
		rank-one projections $P,Q,R\in M_2(\C)$ with a strict triangle defect.
		For fixed $t>0$, the map $(A,B)\mapsto d_t^{\Tr}(A,B)$ is continuous on
		$M_2(\C)_+\times M_2(\C)_+$. Hence adding $\varepsilon I_2$ to all
		three projections preserves the strict defect for all sufficiently small
		$\varepsilon>0$, and makes the three matrices positive definite.
		
		In either case, choosing a sufficiently small regularization and relabeling
		the resulting triple as $A_t,B_t,C_t$ proves
		\eqref{eq:matrix-strict-defect}.
	\end{proof}
	
 We next show that every nonabelian semifinite von Neumann algebra has a
	finite-trace corner containing a unital copy of $M_2(\C)$. This allows
	the strict $2\times2$ triangle defects constructed above to be embedded
	into the algebra.
	\begin{lemma}
		\label{lem:finite-M2-corner}
		Let $(\M,\tau)$ be a nonabelian von Neumann algebra with a faithful normal
		semifinite trace $\tau$. Then there exist a projection $f\in\M$ with
		$\tau(f)<\infty$, a unital $*$-monomorphism
		\[
		\iota:M_2(\C)\longrightarrow f\M f,
		\]
		and a constant $c>0$ such that
		\begin{equation}\label{eq:trace-on-M2-corner}
			\tau(\iota(X))=c\,\Tr(X),
			\qquad X\in M_2(\C).
		\end{equation}
	\end{lemma}
	
	\begin{proof}
   Let
	\[
	\mathcal{P}_\tau
	:=
	\{e\in\M:e=e^*=e^2,\ \tau(e)<\infty\}
	\]
	be the set of finite-trace projections in $\M$. For
	$e,g\in\mathcal{P}_\tau$, let $e\vee g$ denote their join, namely,
	the smallest projection in $\M$ dominating both $e$ and $g$.
	Since
	\[
	\tau(e\vee g)
	\leq
	\tau(e)+\tau(g)
	<\infty,
	\]
	we have $e\vee g\in\mathcal{P}_\tau$. Thus
	$\mathcal{P}_\tau$ is directed under the usual order on projections.
	
	Set
	\[
	z
	:=
	\bigvee_{e\in\mathcal{P}_\tau}e,
	\]
	where the supremum is taken in the projection lattice of $\M$. We
	claim that $z=\one$. Suppose, to the contrary, that $z<\one$.
	Since $\tau$ is semifinite and $\one-z$ is a nonzero positive
	element, there exists a nonzero $y\in\M_+$ such that
	\[
	y\leq\one-z
	\qquad\text{and}\qquad
	\tau(y)<\infty.
	\]
	Choose $0<\delta<\lVert y\rVert$. Then the spectral projection
	\[
	e_\delta
	:=
	\one_{[\delta,\infty)}(y)
	\]
	is nonzero. Moreover, since $0\leq y\leq\one-z$ and $\one-z$ is a
	projection, we have
	\[
	y=(\one-z)y(\one-z),
	\]
	and hence
	\[
	e_\delta\leq\one-z.
	\]
	Functional calculus applied to the scalar inequality
	\[
	\lambda
	\geq
	\delta\one_{[\delta,\infty)}(\lambda),
	\qquad \lambda\geq0,
	\]
	also gives
	\[
	y\geq\delta e_\delta.
	\]
	Therefore,
	\[
	\tau(e_\delta)
	\leq
	\delta^{-1}\tau(y)
	<\infty,
	\]
	so $e_\delta\in\mathcal{P}_\tau$. By the definition of $z$, we must
	also have $e_\delta\leq z$. Consequently,
	\[
	e_\delta\leq z
	\qquad\text{and}\qquad
	e_\delta\leq\one-z,
	\]
	which forces $e_\delta=0$, contradicting its choice. Hence
	$z=\one$. 

		Suppose every finite corner $e\M e$, $e\in\mathcal P_\tau$, were
		abelian. Then, for arbitrary $X,Y\in\M$,
		\[
		[eXe,eYe]=0,
		\qquad e\in\mathcal P_\tau.
		\]
		Along the directed net $\mathcal P_\tau$, one has $e\uparrow\one$
		strongly and hence
		\[
		eXe\longrightarrow X,
		\qquad
		eYe\longrightarrow Y
		\]
		strongly. Both nets are uniformly bounded, so multiplication is strongly
		continuous along them and $[eXe,eYe]\to[X,Y]$ strongly. Thus
		$[X,Y]=0$, contrary to nonabelianness of $\M$. Consequently, some
		finite corner $e\M e$ is nonabelian.

		Projections generate $e\M e$ as a von Neumann algebra, so there exist
		noncommuting projections $a,b\in e\M e$. Put
		\[
		x:=ab(e-a).
		\]
		Then $x\ne0$: otherwise $ab=aba$, while taking adjoints gives
		$ba=aba$, and hence $ab=ba$. Let $x=v|x|$ be its polar
		decomposition. Its initial and final projections satisfy
		\[
		p:=v^*v\le e-a,
		\qquad
		q:=vv^*\le a.
		\]
		Thus $p$ and $q$ are nonzero orthogonal equivalent projections. Set
		$f=p+q$. The assignments
		\[
		E_{11}\mapsto q,
		\qquad
		E_{22}\mapsto p,
		\qquad
		E_{12}\mapsto v,
		\qquad
		E_{21}\mapsto v^*
		\]
		form a system of matrix units with identity $f$, and therefore define a
		unital $*$-monomorphism
		\[
		\iota:M_2(\C)\longrightarrow f\M f.
		\]
	  Since $f\leq e$, one has $\tau(f)<\infty$. The functional
		$\tau\circ\iota$ is therefore a finite trace on $M_2(\C)$. Moreover,
		\[
		\tau(p)
		=
		\tau(v^*v)
		=
		\tau(vv^*)
		=
		\tau(q)
		=:c,
		\]
		where $0<c<\infty$ because $p$ and $q$ are nonzero projections
		dominated by $e$ and $\tau$ is faithful. Since every trace on
		$M_2(\C)$ is a scalar multiple of the usual matrix trace, it follows
		that
		\[
		\tau(\iota(X))
		=
		c\,\Tr(X),
		\qquad X\in M_2(\C).
		\]
	\end{proof}
	
  \begin{proof}[Proof of Theorem~\ref{thm:classification}\textup{(ii)}]
 		For $1\leq t\leq2$, metricity follows from
 		Theorem~\ref{thm:classification}\textup{(i)}.
 		
 		Suppose first that $\M$ is abelian. Then measurable functional
 		calculus gives, for every $t>0$,
 		\[
 		\left(
 		A^{t/4}B^{t/2}A^{t/4}
 		\right)^{1/t}
 		=
 		\left(A^{t/2}B^{t/2}\right)^{1/t}
 		=
 		A^{1/2}B^{1/2}.
 		\]
 		Hence
 		\[
 		Q_t(A,B)=\tau(A^{1/2}B^{1/2})
 		\]
 		and
 		\[
 		d_t(A,B)^2
 		=
 		\frac12
 		\|A^{1/2}-B^{1/2}\|_2^2.
 		\]
 		Thus $d_t$ is a metric for every $t>0$.
 		
 		Conversely, suppose that $\M$ is nonabelian and
 		$t\notin[1,2]$. By Lemma~\ref{lem:finite-M2-corner}, there exist
 		a finite-trace projection $f\in\M$, a unital $*$-monomorphism
 		\[
 		\iota:M_2(\C)\longrightarrow f\M f,
 		\]
 		and a constant $c>0$ such that
 		\[
 		\tau(\iota(X))=c\,\Tr(X),
 		\qquad X\in M_2(\C).
 		\]
 		Let $A_t,B_t,C_t\in M_2(\C)^{++}$ be the triple given by
 		Proposition~\ref{prop:matrix-obstructions}, and set
 		\[
 		\widehat A_t:=\iota(A_t),\qquad
 		\widehat B_t:=\iota(B_t),\qquad
 		\widehat C_t:=\iota(C_t).
 		\]
 		
 		Since $\iota$ preserves continuous functional calculus, for all
 		$X,Y\in M_2(\C)_+$,
 		\[
 		Q_t^\tau(\iota(X),\iota(Y))
 		=
 		c\,Q_t^{\Tr}(X,Y).
 		\]
 		Together with $\tau(\iota(X))=c\,\Tr(X)$, this gives
 		\[
 		d_t^\tau(\iota(X),\iota(Y))^2
 		=
 		c\,d_t^{\Tr}(X,Y)^2.
 		\]
 		Consequently,
 		\[
 		d_t^\tau(\widehat A_t,\widehat B_t)
 		>
 		d_t^\tau(\widehat A_t,\widehat C_t)
 		+
 		d_t^\tau(\widehat C_t,\widehat B_t).
 		\]
 		Thus $d_t$ fails the triangle inequality on
 		$L^1(\M,\tau)_+$.
 		
 		The three transported elements are bounded and invertible in the
 		corner $f\M f$; in particular, their support projection is $f$,
 		which has finite trace.
 		
 		Finally, suppose that $\tau(\one)<\infty$. For $X\in
 		\{\widehat A_t,\widehat B_t,\widehat C_t\}$, set
 		\[
 		X_\varepsilon:=X+\varepsilon(\one-f),
 		\qquad \varepsilon>0.
 		\]
 		Each $X_\varepsilon$ is invertible in $\M$. Since the two
 		summands are supported on the orthogonal projections $f$ and
 		$\one-f$, for every pair $X,Y$ among the transported elements,
 		\[
 		Q_t^\tau(X_\varepsilon,Y_\varepsilon)
 		=
 		Q_t^\tau(X,Y)+\varepsilon\tau(\one-f).
 		\]
 		The same term $\varepsilon\tau(\one-f)$ is added to the
 		arithmetic part of the squared distance, and therefore
 		\[
 		d_t^\tau(X_\varepsilon,Y_\varepsilon)
 		=
 		d_t^\tau(X,Y).
 		\]
 		Hence the same strict triangle defect is realized by positive
 		invertible elements of $\M$.
 \end{proof}
	
	Before proving the $C^*$-algebra assertion, we record the equivalent
	parametrization used for the
	symmetric $\alpha$-$z$ Bures--Wasserstein expression. For $z>0$, set
	\begin{equation}\label{eq:z-distance}
		D_z(A,B):=\left[
		\frac{\tau(A)+\tau(B)}2
		-\tau\!\left(
		\left(A^{1/(4z)}B^{1/(2z)}A^{1/(4z)}\right)^z
		\right)
		\right]^{1/2}.
	\end{equation}
	
	The substitution $t=1/z$ gives the equivalent sharp classification.
	
	\begin{corollary}
		\label{cor:z-classification}
		Let $(\M,\tau)$ be as in Theorem~\ref{thm:classification}\textup{(ii)}. Then
		\begin{equation}\label{eq:sharp-z-classification}
			D_z\text{ is a metric on }L^1(\M,\tau)_+
			\quad\Longleftrightarrow\quad
			z\in[1/2,1]\ \text{or}\ \M\text{ is abelian}.
		\end{equation}
	\end{corollary}
	
	\begin{proof}
		The definitions give $D_z=d_{1/z}$, so the assertion follows from
		Theorem~\ref{thm:classification}\textup{(ii)}.
	\end{proof}
	
 \begin{proof}[Proof of Theorem~\ref{thm:classification}\textup{(iii)}]
		Let
		\[
		\M_\tau:=\pi_\tau(\A)'',
		\]
		and let $\bar\tau$ be the canonical faithful normal finite extension
		of $\tau$ to $\M_\tau$. Since $\tau$ is faithful, the GNS
		representation $\pi_\tau$ is faithful, and we identify $\A$ with
		its image in $\M_\tau$. For $A,B\in\A_+$, continuous functional
		calculus is preserved by this identification and
		$\bar\tau|_{\A}=\tau$; hence
		\[
		Q_t^\tau(A,B)=Q_t^{\bar\tau}(A,B).
		\]
		
For $1\leq t\leq2$,
the metricity of $d_t^\tau$ on $\A_+$
follows by restriction from
Theorem~\ref{thm:classification}\textup{(i)}.  If $\A$ is
		abelian, then, for every $t>0$,
		\[
		\left(A^{t/4}B^{t/2}A^{t/4}\right)^{1/t}
		=
		A^{1/2}B^{1/2},
		\]
		and therefore
		\[
		d_t^\tau(A,B)
		=
		2^{-1/2}\|A^{1/2}-B^{1/2}\|_{2,\tau}.
		\]
		Thus $d_t^\tau$ is a metric.
		
		Conversely, suppose that $\A$ is nonabelian and
		$t\notin[1,2]$. Since $\pi_\tau$ is faithful,
		$\M_\tau$ is nonabelian. Moreover,
		$\bar\tau(\one)=\tau(\one)<\infty$. Hence the strengthened
		assertion in
		Theorem~\ref{thm:classification}\textup{(ii)} yields
		$X_1,X_2,X_3\in(\M_\tau)^{++}$ such that
		\[
		\Delta
		:=
		d_t^{\bar\tau}(X_1,X_2)
		-d_t^{\bar\tau}(X_1,X_3)
		-d_t^{\bar\tau}(X_3,X_2)
		>0.
		\]
		
By Kaplansky's density theorem and the strong continuity of bounded
continuous functional calculus
\cite[Theorems~4.3.2 and~4.3.3]{Murphy}, for each
$i\in\{1,2,3\}$ there exists a directed net
$(a_{i,\lambda_i})_{\lambda_i\in\Lambda_i}\subseteq\A_+$ such that
\[
0\leq a_{i,\lambda_i}\leq\|X_i\|\one,
\qquad
a_{i,\lambda_i}\longrightarrow X_i
\quad\text{strongly}.
\]
		Passing to the product directed set
		\[
		\Lambda:=\Lambda_1\times\Lambda_2\times\Lambda_3,
		\]
		we may regard these as three nets
		$(a_{i,\lambda})_{\lambda\in\Lambda}$ with a common index set.
		
		Fix $i,j$. Since the nets are uniformly bounded, continuous
		functional calculus, or equivalently polynomial approximation on a
		common compact interval, gives
		\[
		a_{i,\lambda}^{t/4}\longrightarrow X_i^{t/4},
		\qquad
		a_{j,\lambda}^{t/2}\longrightarrow X_j^{t/2}
		\quad\text{strongly}.
		\]
		Multiplication is jointly strongly continuous on uniformly bounded
		sets, and hence
		\[
		Y_{ij,\lambda}
		:=
		a_{i,\lambda}^{t/4}
		a_{j,\lambda}^{t/2}
		a_{i,\lambda}^{t/4}
		\longrightarrow
		Y_{ij}
		:=
		X_i^{t/4}X_j^{t/2}X_i^{t/4}
		\quad\text{strongly}.
		\]
		The positive operators $Y_{ij,\lambda}$ are uniformly bounded.
		Applying continuous functional calculus once more, now to
		$x\mapsto x^{1/t}$, yields
		\[
		Y_{ij,\lambda}^{1/t}
		\longrightarrow
		Y_{ij}^{1/t}
		\quad\text{strongly}.
		\]
		
		Uniformly bounded strong convergence implies ultraweak convergence.
		Since $\bar\tau$ is normal,
		\[
		\bar\tau(a_{i,\lambda})
		\longrightarrow
		\bar\tau(X_i)
		\]
		and
		\[
		Q_t^\tau(a_{i,\lambda},a_{j,\lambda})
		=
		Q_t^{\bar\tau}(a_{i,\lambda},a_{j,\lambda})
		\longrightarrow
		Q_t^{\bar\tau}(X_i,X_j).
		\]
		Consequently,
		\[
		d_t^\tau(a_{i,\lambda},a_{j,\lambda})
		\longrightarrow
		d_t^{\bar\tau}(X_i,X_j)
		\]
		for every $i,j$. It follows that
		\[
		\begin{aligned}
			&d_t^\tau(a_{1,\lambda},a_{2,\lambda})
			-d_t^\tau(a_{1,\lambda},a_{3,\lambda})
			-d_t^\tau(a_{3,\lambda},a_{2,\lambda})\\
			&\hspace{35mm}\longrightarrow\Delta>0.
		\end{aligned}
		\]
		Thus, for every sufficiently large $\lambda$, the strict triangle
		defect persists in $\A_+$. Hence $d_t^\tau$ is not a metric on
		$\A_+$.
\end{proof}
	
\begin{remark}
	Applying Theorem~\ref{thm:classification} to
	$B(\mathcal H)$ with its canonical trace shows that, if
	$\dim\mathcal H\ge2$, then $d_t$ is a complete metric on the
	positive trace-class cone $\mathcal S_1(\mathcal H)_+$ exactly when
	$1\le t\le2$. If $\dim\mathcal H=1$, then every $t>0$ yields the
	Hellinger metric.
\end{remark}
	
	\section{The known logarithmic endpoint obstruction}
	\label{sec:zero-endpoint}
	
	At $t=0$, the family is defined by the symmetrized Lie--Trotter limit
	rather than by the power formula. For $A,B\in M_n(\C)^{++}$, set
	\begin{equation}\label{eq:zero-affinity-distance}
		\begin{aligned}
			Q_0^{\Tr}(A,B)
			&:=\Tr\exp\!\left(\frac{\log A+\log B}{2}\right),\\
			d_0^{\Tr}(A,B)^2
			&:=\frac{\Tr A+\Tr B}{2}-Q_0^{\Tr}(A,B).
		\end{aligned}
	\end{equation}
	Indeed, the symmetrized Lie--Trotter formula gives
	\[
	\left(A^{t/4}B^{t/2}A^{t/4}\right)^{1/t}
	\longrightarrow
	\exp\!\left(\frac{\log A+\log B}{2}\right)
	\qquad(t\downarrow0).
	\]
The radicand defining $d_0^{\Tr}$ in
\eqref{eq:zero-affinity-distance} is nonnegative. Indeed, the
Golden--Thompson inequality
\cite[Eq.~(21)]{BhatiaGaubertJain}, followed by the
Hilbert--Schmidt Cauchy--Schwarz and arithmetic--geometric mean
inequalities, gives
\[
Q_0^{\Tr}(A,B)
\leq \Tr(A^{1/2}B^{1/2})
\leq \sqrt{\Tr A\,\Tr B}
\leq \frac{\Tr A+\Tr B}{2}.
\]
	We keep this endpoint separate from the semifinite $L^1$-theory, since
	the logarithms in \eqref{eq:zero-affinity-distance} are not available for
	arbitrary noninvertible $L^1$-densities.
	
The failure of the triangle inequality at this endpoint was already proved
by Bhatia, Gaubert, and Jain \cite[Section~2]{BhatiaGaubertJain}.  Their log-Euclidean distance
	\[
	\delta_{\rm LE}(A,B)
	:=
	\left[
	\Tr(A+B)-2\Tr\exp\!\left(\frac{\log A+\log B}{2}\right)
	\right]^{1/2}
	\]
	satisfies $\delta_{\rm LE}=\sqrt2\,d_0^{\Tr}$, and their explicit
	positive definite $2\times2$ example has a strict triangle defect
	\cite[Section~2]{BhatiaGaubertJain}.
	
	The known $2\times2$ log-Euclidean example extends to every matrix size.
	
	\begin{proposition}
		\label{prop:zero-endpoint}
		For every integer $n\ge2$, there exist
		$A,B,C\in M_n(\C)^{++}$ such that
		\begin{equation}\label{eq:zero-strict-defect}
			d_0^{\Tr}(A,B)
			>d_0^{\Tr}(A,C)+d_0^{\Tr}(C,B).
		\end{equation}
		Thus the log-Euclidean endpoint fails the triangle inequality already on
		the positive definite cone.
	\end{proposition}
	
	\begin{proof}
		For $n=2$, this is the log-Euclidean example of
		\cite[Section~2]{BhatiaGaubertJain}, after division by $\sqrt2$.
		For $n>2$, take direct sums of the three matrices with the same block
		$\lambda I_{n-2}$, where $\lambda>0$. Functional calculus respects
		direct sums, and the common block contributes zero to each squared
		distance.
	\end{proof}
	
	We can now state the resulting resolution of
Problem~2 in \cite{KomalovicsMolnar}.
	
	\begin{corollary}
		\label{cor:KM-problem2}
		Combining the preceding known endpoint obstruction with the
		positive-parameter results of this paper, the three questions in
		\cite[Problem~2]{KomalovicsMolnar} have affirmative answers. More precisely:
		\begin{enumerate}
			\item for every $n\ge2$, $d_t^{\Tr}$ is a metric on
			$M_n(\C)_+$ whenever $1<t<2$;
			\item for every $n\ge2$, it is not a metric whenever $0<t<1$, and
			at $t=0$ the log-Euclidean formula fails the triangle inequality
			already on $M_n(\C)^{++}$; hence no extension of that formula to
			$M_n(\C)_+$ can be a metric;
			\item for every unital $C^*$-algebra with a faithful tracial positive
			linear functional $\tau$, $d_t^\tau$ is a metric on $\A_+$
			whenever $1<t<2$.
		\end{enumerate}
	\end{corollary}
	
	\begin{proof}
		The positive-parameter matrix assertions follow from
		Theorem~\ref{thm:classification}\textup{(ii)}; the endpoint assertion is the result of
		Bhatia, Gaubert, and Jain recorded in
		Proposition~\ref{prop:zero-endpoint}; and the general $C^*$-algebra
		assertion is Theorem~\ref{thm:classification}\textup{(iii)}.
	\end{proof}
	
	\section{Further consequences: data processing and metric completions}
	\label{sec:cores}

	\subsection{Data processing}
	
		We recall the quantity $Q_{\alpha,z}$ appearing in the definition
		of the $\alpha$-$z$ R\'enyi divergence.
		For $A,B\in L^1(\mathcal M,\tau)_+$, let
		\[
		\varphi_A(x):=\tau(Ax),\qquad
		\varphi_B(x):=\tau(Bx),\qquad x\in\mathcal M,
		\]
		be the associated normal positive functionals.
		For $0<\alpha<1$, $z>0$, and $B\ne0$, the quantity
		$Q_{\alpha,z}$ in \cite[Definition~2.1]{HiaiJencova}
		has the following trace-density expression:
		\[
		Q_{\alpha,z}(\varphi_B\Vert\varphi_A)
		=
		\tau\!\left(
		\left(
		A^{(1-\alpha)/(2z)}
		B^{\alpha/z}
		A^{(1-\alpha)/(2z)}
		\right)^z
		\right).
		\]
		This semifinite tracial realization is explained at the beginning
		of \cite[Section~5.1]{HiaiJencova}. No trace-one normalization
		of $A$ or $B$ is assumed.
		
		In particular, taking $\alpha=1/2$ and $z=1/t$ gives
		\[
		Q_{1/2,1/t}(\varphi_B\Vert\varphi_A)
		=
		\tau\!\left(
		\left(
		A^{t/4}B^{t/2}A^{t/4}
		\right)^{1/t}
		\right)
		=
		Q_t^\tau(A,B).
		\]
		We use this identification and the data-processing inequality
		for $D_{\alpha,z}$ to obtain the following contractivity statement.

	\begin{proposition}
		\label{prop:channel-contraction}
		Let $(\M,\tau_\M)$ and $(\N,\tau_\N)$ be von Neumann algebras
		equipped with faithful normal semifinite traces, and let
		$\Gamma:\N\to\M$ be a normal positive unital map. Identify the
		positive preduals with the corresponding $L^1$-cones, and let
		$\Gamma_*:L^1(\M,\tau_\M)_+\to L^1(\N,\tau_\N)_+$
		be the predual map. Then, for $1\le t\le2$
		and all $A,B\in L^1(\M,\tau_\M)_+$,
		\begin{equation}\label{eq:channel-contraction}
			d_t^{\tau_\N}(\Gamma_*A,\Gamma_*B)
			\le d_t^{\tau_\M}(A,B).
		\end{equation}
		In particular, Schr\"odinger-picture normal completely positive
		trace-preserving maps contract $d_t$.
	\end{proposition}
	
	\begin{proof}
		
			Fix $1\le t\le2$ and $A,B\in L^1(\M,\tau_\M)_+$, and let
			$\varphi_A,\varphi_B$ be the associated normal positive
			functionals on $\M$.
		
		Under the canonical identification
		$L^1(\M,\tau_\M)_+\cong\M_*^+$, the element
		$\Gamma_*A$ is characterized by
		\[
		\varphi_{\Gamma_*A}=\varphi_A\circ\Gamma.
		\]
		Since $\Gamma$ is unital,
		\[
		\tau_\N(\Gamma_*A)
		=(\varphi_A\circ\Gamma)(\one_\N)
		=\varphi_A(\one_\M)
		=\tau_\M(A),
		\]
		and similarly for $B$.
		
		If $Q_t^{\tau_\M}(A,B)=0$,
		the nonnegativity of
			$Q_t^{\tau_\N}(\Gamma_*A,\Gamma_*B)$
		and preservation of mass give
		\[
		\begin{aligned}
			d_t^{\tau_\N}(\Gamma_*A,\Gamma_*B)^2
			&\le
			\frac{\tau_\N(\Gamma_*A)+\tau_\N(\Gamma_*B)}2\\
			&=
			\frac{\tau_\M(A)+\tau_\M(B)}2
			=
			d_t^{\tau_\M}(A,B)^2.
		\end{aligned}
		\]
		This includes the cases in which one input is zero.
		Henceforth assume $Q_t^{\tau_\M}(A,B)>0$; in particular,
		$A$ and $B$ are nonzero.
		
		Set
		\[
		\alpha:=\frac12,\qquad z:=\frac1t.
		\]
		Since $1\le t\le2$, we have
		\[
		z\ge\frac12=\max\{\alpha,1-\alpha\},
		\]
		so the data-processing inequality in
		\cite[Theorem~3.3(i)]{HiaiJencova}
		applies.
		
			By the identification above,
			\[
			Q_{1/2,1/t}(\varphi_B\Vert\varphi_A)
			=
			Q_t^{\tau_\M}(A,B),
			\]
		
		and similarly
		\[
		Q_{1/2,1/t}
		(\varphi_B\circ\Gamma\Vert\varphi_A\circ\Gamma)
		=
		Q_t^{\tau_\N}(\Gamma_*A,\Gamma_*B).
		\]
		
		The data-processing inequality states that
		\[
		D_{1/2,1/t}
		(\varphi_B\circ\Gamma\Vert\varphi_A\circ\Gamma)
		\le
		D_{1/2,1/t}(\varphi_B\Vert\varphi_A).
		\]
		
			By \cite[Definition~2.1]{HiaiJencova}, for normal positive
			functionals $\psi,\varphi$ with $\psi\ne0$,
		
		\[
		D_{1/2,1/t}(\psi\Vert\varphi)
		=
		-2\log
		\frac{Q_{1/2,1/t}(\psi\Vert\varphi)}{\psi(\one)},
		\]
		
			with the convention $-\log0=+\infty$.
			Since $Q_t^{\tau_\M}(A,B)>0$, the right-hand side of the
			data-processing inequality is finite. Consequently,
			\[
			Q_{1/2,1/t}
			(\varphi_B\circ\Gamma\Vert\varphi_A\circ\Gamma)>0.
			\]
		
		Since
		\[
		(\varphi_B\circ\Gamma)(\one_\N)
		=
		\varphi_B(\one_\M)>0,
		\]
		
			the two logarithmic expressions have the same positive
			denominator. As $x\mapsto-2\log x$ is strictly decreasing
			on $(0,\infty)$, the data-processing inequality is equivalent to
		
		\[
		Q_t^{\tau_\N}(\Gamma_*A,\Gamma_*B)
		\ge
		Q_t^{\tau_\M}(A,B).
		\]
		Therefore,
		\begin{align*}
			d_t^{\tau_\N}(\Gamma_*A,\Gamma_*B)^2
			&=
			\frac{\tau_\N(\Gamma_*A)+\tau_\N(\Gamma_*B)}2
			-
			Q_t^{\tau_\N}(\Gamma_*A,\Gamma_*B)\\
			&\le
			\frac{\tau_\M(A)+\tau_\M(B)}2
			-
			Q_t^{\tau_\M}(A,B)\\
			&=
			d_t^{\tau_\M}(A,B)^2.
		\end{align*}
		Taking square roots proves
		\eqref{eq:channel-contraction}.
	\end{proof}

	\subsection{Finite-support and \texorpdfstring{$C^*$}{C*-algebraic} cores}
	
	Let
	\begin{equation}\label{eq:finite-support-core}
		\mathfrak f_{\tau,+}
		:=\{A\in\M_+:\tau(s(A))<\infty\},
	\end{equation}
	where $s(A)$ is the support projection of $A$.
	
	The finite-support cone is a metric core, including on every trace level.
	
	\begin{proposition}
		\label{prop:finite-support-core}
		For $1\le t\le2$, the cone $\mathfrak f_{\tau,+}$ is dense in
		$L^1(\M,\tau)_+$ for both $\|\cdot\|_1$ and $d_t$. Hence its metric
		completion is the full noncommutative $L^1$-cone. The state-density space
		\begin{equation}\label{eq:state-space}
			\mathsf S_1(\M,\tau)
			:=\{\rho\in L^1(\M,\tau)_+:\tau(\rho)=1\}
		\end{equation}
		is closed and complete, and its bounded finite-support states are dense.
	\end{proposition}
	
 \begin{proof}
		Let $A\in L^1(\M,\tau)_+$, and set
		\[
		e_n:=\one_{[1/n,n]}(A),
		\qquad
		A_n:=Ae_n=e_nAe_n.
		\]
		Since
		\[
		\frac1n e_n\leq A_n\leq ne_n,
		\]
		the operator $A_n$ is bounded and $s(A_n)=e_n$. Moreover,
		\[
		\frac1n\tau(e_n)
		\leq\tau(A_n)
		\leq\tau(A),
		\]
		and hence
		\[
		\tau(s(A_n))
		=\tau(e_n)
		\leq n\tau(A)<\infty.
		\]
		Thus $A_n\in\mathfrak f_{\tau,+}$.
		
		The projections $e_n$ increase to $s(A)$, and
		$0\leq A_n\leq A$. Therefore, by normality of $\tau$,
		\[
		\|A-A_n\|_1
		=
		\tau(A-A_n)
		\longrightarrow0.
		\]
		Using \eqref{eq:endpoint-bounds} and
		\eqref{eq:Powers-Stormer-bound}, we obtain
		\[
		d_t(A_n,A)^2
		\leq d_1(A_n,A)^2
		\leq\frac12\|A_n-A\|_1
		\longrightarrow0.
		\]
		Hence $\mathfrak f_{\tau,+}$ is dense in
		$L^1(\M,\tau)_+$ for both $\|\cdot\|_1$ and $d_t$.
		Since $L^1(\M,\tau)_+$ is complete with respect to $d_t$ by
		Theorem~\ref{thm:classification}\textup{(i)}, it is the metric
		completion of $\mathfrak f_{\tau,+}$.
		
		Now suppose that $\tau(A)=1$. Then
		\[
		\tau(A_n)\longrightarrow1.
		\]
		For all sufficiently large $n$, define
		\[
		\rho_n:=\frac{A_n}{\tau(A_n)}.
		\]
		Each $\rho_n$ is a bounded finite-support state, and
		\begin{align*}
			\|\rho_n-A\|_1
			&\leq
			\left\|
			\frac{A_n}{\tau(A_n)}-A_n
			\right\|_1
			+\|A_n-A\|_1\\
			&=
			|1-\tau(A_n)|
			+\|A_n-A\|_1
			\longrightarrow0.
		\end{align*}
		Consequently, $d_t(\rho_n,A)\to0$.
		
		Finally, the trace functional is continuous on $L^1(\M,\tau)$;
		therefore $\mathsf S_1(\M,\tau)$ is $L^1$-closed. Since the
		$L^1$- and $d_t$-topologies coincide and
		$L^1(\M,\tau)_+$ is $d_t$-complete by
		Theorem~\ref{thm:classification}\textup{(i)}, the state-density
		space $\mathsf S_1(\M,\tau)$ is closed and complete.
\end{proof}

	Let $(\M,\tau)$ be a von Neumann algebra equipped with a faithful
	normal semifinite trace, and let $\A\subseteq\M$ be a $C^*$-subalgebra
	such that the restriction of $\tau$ to $\A$ is densely defined and
	semifinite. Here densely defined means that
	\[
	\overline{\{a\in\A_+:\tau(a)<\infty\}}^{\,\|\cdot\|}
	=\A_+,
	\]
	where $\|\cdot\|$ is the $C^*$-norm; see
	\cite[Section~1.1]{KustermansVaesWeights}.

Define
	\begin{equation}\label{eq:Cstar-trace-ideals}
		\A_{\tau,+}:=\{a\in\A_+:\tau(a)<\infty\},
		\qquad
		\mathfrak n_\tau(\A):=\{a\in\A:\tau(a^*a)<\infty\}.
	\end{equation}
	When $\tau(\one)=\infty$, the natural metric domain is
	$\A_{\tau,+}$, not all of $\A_+$.
	
	Its completion is governed exactly by an $L^2$-density criterion.
	
	\begin{theorem}
		\label{thm:L2-core}
		Let $1\le t\le2$. The completion of $(\A_{\tau,+},d_t)$ is canonically
		identified with
		\begin{equation}\label{eq:Cstar-completion}
			\left(\overline{\A_{\tau,+}}^{\,\|\cdot\|_1},d_t\right)
			\subseteq L^1(\M,\tau)_+.
		\end{equation}
		It equals the full cone $L^1(\M,\tau)_+$ if and only if
		\begin{equation}\label{eq:L2-core-condition}
			\overline{\mathfrak n_\tau(\A)}^{\,\|\cdot\|_2}
			=L^2(\M,\tau).
		\end{equation}
		Condition \eqref{eq:L2-core-condition} is therefore the exact
		Hilbert-space density requirement for obtaining the full $L^1$-cone.
	\end{theorem}

	\begin{proof}
		Set
		\[
		\mathcal C:=\overline{\A_{\tau,+}}^{\,\|\cdot\|_1}
		\subseteq L^1(\M,\tau)_+.
		\]
		By Theorem~\ref{thm:classification}\textup{(i)}, the
		$L^1$- and $d_t$-topologies on the positive cone coincide,
		and the cone is complete for $d_t$. Hence $\mathcal C$ is
		$d_t$-closed and complete, with $\A_{\tau,+}$ dense in it.
		The inclusion of $\A_{\tau,+}$ therefore extends uniquely to
		an isometry from its metric completion onto $\mathcal C$.
		This proves \eqref{eq:Cstar-completion}.
		
		Suppose first that \eqref{eq:L2-core-condition} holds.
		Given $A\in L^1(\M,\tau)_+$, set $h=A^{1/2}\in L^2(\M,\tau)_+$
		and choose $b_n\in\mathfrak n_\tau(\A)$ with
		$\|b_n-h\|_2\to0$. Then
		$b_n^*b_n\in\A_{\tau,+}$, and H\"older's inequality gives
		\[
		\begin{aligned}
			\|b_n^*b_n-A\|_1
			&=\|b_n^*(b_n-h)+(b_n^*-h)h\|_1\\
			&\leq(\|b_n\|_2+\|h\|_2)\|b_n-h\|_2
			\longrightarrow0.
		\end{aligned}
		\]
		Thus $\mathcal C=L^1(\M,\tau)_+$.
		
		Conversely, suppose that $\mathcal C=L^1(\M,\tau)_+$.
		For $h\in L^2(\M,\tau)_+$, choose $a_n\in\A_{\tau,+}$
		such that $\|a_n-h^2\|_1\to0$. Continuous functional calculus
		gives $a_n^{1/2}\in\A$; this remains valid for nonunital
		$\A$, since the square-root function vanishes at zero.
		Moreover,
		\[
		\tau\bigl((a_n^{1/2})^*a_n^{1/2}\bigr)=\tau(a_n)<\infty,
		\]
		so $a_n^{1/2}\in\mathfrak n_\tau(\A)$.
		The Powers--St\o rmer estimate \eqref{eq:Powers-Stormer-bound}
		yields
		\[
		\|a_n^{1/2}-h\|_2^2
		\leq\|a_n-h^2\|_1\longrightarrow0.
		\]
		Thus the $L^2$-closure of $\mathfrak n_\tau(\A)$ contains
		$L^2(\M,\tau)_+$. The space $\mathfrak n_\tau(\A)$ is
		linear: for $a,b\in\mathfrak n_\tau(\A)$,
		\[
		(a+b)^*(a+b)\leq2a^*a+2b^*b,
		\]
		and scalar multiplication also preserves finite $L^2$-norm.
		Finally, every $x\in L^2(\M,\tau)$ has the decomposition
		\[
		x=(\operatorname{Re}x)_+-(\operatorname{Re}x)_-
		+i\bigl((\operatorname{Im}x)_+-(\operatorname{Im}x)_-\bigr),
		\]
		whose four positive parts belong to $L^2(\M,\tau)_+$.
		Their complex linear span is therefore all of $L^2(\M,\tau)$,
		which proves \eqref{eq:L2-core-condition}.
	\end{proof}

	\subsection{Trace-GNS completion and faithful densities}

	For the trace-GNS von Neumann envelope, the $L^2$-density condition holds
	automatically.
	
	\begin{corollary}
		\label{cor:trace-GNS}

	Let $\tau$ be a faithful, densely defined, lower semicontinuous,
	semifinite trace on a $C^*$-algebra $\A$, where densely defined means
	that
	\[
	\overline{\A_{\tau,+}}^{\,\|\cdot\|}=\A_+,
	\]
	with $\|\cdot\|$ denoting the $C^*$-norm.
  Let
		\[
		\M_\tau:=\pi_\tau(\A)'',
		\]
		and let $\bar\tau$ be the canonical faithful normal semifinite extension.
		For every $1\le t\le2$, the metric completion of
		$(\A_{\tau,+},d_t^\tau)$ is canonically isometric to
		\begin{equation}\label{eq:trace-GNS-completion}
			\bigl(L^1(\M_\tau,\bar\tau)_+,d_t^{\bar\tau}\bigr).
		\end{equation}
		The completion of the trace-one part is
		\begin{equation}\label{eq:trace-GNS-states}
			\{\rho\in L^1(\M_\tau,\bar\tau)_+:\bar\tau(\rho)=1\}.
		\end{equation}
	\end{corollary}
	
  \begin{proof}
		Let
		\[
		(\mathcal H_\tau,\pi_\tau,\Lambda_\tau)
		\]
		be the GNS construction associated with $\tau$. Since $\tau$
		is faithful, $\pi_\tau$ is faithful, and we identify $\A$
		with its image $\pi_\tau(\A)\subseteq\M_\tau$.
		
	The case $\A=\{0\}$ is immediate, so assume $\A\ne\{0\}$.
	The trace $\tau$ is a KMS weight with modular group
	$\sigma_s=\mathrm{id}_{\A}$, $s\in\mathbb R$;
	see \cite[Definition~1.10]{KustermansVaesWeights}.
	Its $W^*$-lift $\bar\tau$ to
	$\M_\tau=\pi_\tau(\A)''$ satisfies
	\[
	\bar\tau\circ\pi_\tau=\tau
	\quad\text{on }\A_+
	\]
	by \cite[Definition~2.10]{KustermansVaesWeights}.
	
	By \cite[Proposition~2.22 and the paragraph following its proof]
	{KustermansVaesWeights},
	$\bar\tau$ is faithful, normal, and semifinite, with modular
	group extending $\sigma$.
	This extension is trivial by normality and the ultraweak density
	of $\pi_\tau(\A)$ in $\M_\tau$.
	Hence $\bar\tau$ is a trace by
	\cite[Proposition~2.7]{KustermansVaesWeights}.
	In particular,
	\[
	\bar\tau\!\left(\pi_\tau(a)^*\pi_\tau(a)\right)
	=
	\tau(a^*a),
	\qquad a\in\mathfrak n_\tau(\A).
	\]
	
	By \cite[Proposition~2.16 and its proof]{KustermansVaesWeights},
	using the tracial realization of the GNS space of $\bar\tau$,
	the map
	\[
	\Lambda_\tau(a)\longmapsto\pi_\tau(a),
	\qquad a\in\mathfrak n_\tau(\A),
	\]
	extends to a unitary
	\[
	\mathcal H_\tau\longrightarrow L^2(\M_\tau,\bar\tau).
	\]
	The density of $\Lambda_\tau(\mathfrak n_\tau(\A))$
	in $\mathcal H_\tau$ therefore gives
	\begin{equation}\label{eq:GNS-L2-density}
		\overline{
			\pi_\tau(\mathfrak n_\tau(\A))
		}^{\,\|\cdot\|_{2,\bar\tau}}
		=
		L^2(\M_\tau,\bar\tau).
	\end{equation}
		
		The embedding $\pi_\tau$ also preserves the metric. Indeed, for
		$a,b\in\A_{\tau,+}$, continuous functional calculus gives
		\[
		\begin{aligned}
			&\left(
			\pi_\tau(a)^{t/4}
			\pi_\tau(b)^{t/2}
			\pi_\tau(a)^{t/4}
			\right)^{1/t}\\
			&\qquad=
			\pi_\tau\!\left(
			\left(a^{t/4}b^{t/2}a^{t/4}\right)^{1/t}
			\right).
		\end{aligned}
		\]
		Since $\bar\tau\circ\pi_\tau=\tau$ on $\A_+$, it follows that
		\[
		Q_t^{\bar\tau}(\pi_\tau(a),\pi_\tau(b))
		=
		Q_t^\tau(a,b)
		\]
		and hence
		\[
		d_t^{\bar\tau}(\pi_\tau(a),\pi_\tau(b))
		=
		d_t^\tau(a,b).
		\]
		
	Since $\bar\tau\circ\pi_\tau=\tau$, we have
	\[
	\mathfrak n_{\bar\tau}(\pi_\tau(\A))
	=\pi_\tau(\mathfrak n_\tau(\A)).
	\]
	Consequently, \eqref{eq:GNS-L2-density} is exactly condition
	\eqref{eq:L2-core-condition} for $\pi_\tau(\A)\subseteq\M_\tau$.
	Theorem~\ref{thm:L2-core} therefore gives
	\[
	\overline{\pi_\tau(\A_{\tau,+})}^{\,\|\cdot\|_{1,\bar\tau}}
	=L^1(\M_\tau,\bar\tau)_+.
	\]
	Together with the preceding isometry, the completion assertion
	of that theorem proves \eqref{eq:trace-GNS-completion}.
		
		It remains to consider the trace-one part. Let
		\[
		\mathsf S_\tau(\A)
		:=
		\{a\in\A_{\tau,+}:\tau(a)=1\},
		\]
		and let
		\[
		\rho\in L^1(\M_\tau,\bar\tau)_+,
		\qquad
		\bar\tau(\rho)=1.
		\]
		By the first part, there exist $b_n\in\A_{\tau,+}$ such that
		\[
		\pi_\tau(b_n)\longrightarrow\rho
		\qquad\text{in }L^1(\M_\tau,\bar\tau).
		\]
		Set
		\[
		m_n:=\tau(b_n)
		=\bar\tau(\pi_\tau(b_n)).
		\]
		The continuity of the mass functional on $L^1$ gives
		$m_n\to1$. Hence $m_n>0$ for all sufficiently large $n$.
		For such $n$, put
		\[
		c_n:=\frac{b_n}{m_n}\in\mathsf S_\tau(\A).
		\]
		Then
		\begin{align*}
			\|\pi_\tau(c_n)-\rho\|_1
			&\leq
			\left\|
			\frac{\pi_\tau(b_n)}{m_n}-\pi_\tau(b_n)
			\right\|_1
			+
			\|\pi_\tau(b_n)-\rho\|_1\\
			&=
			|1-m_n|
			+
			\|\pi_\tau(b_n)-\rho\|_1
			\longrightarrow0.
		\end{align*}
		By Theorem~\ref{thm:classification}\textup{(i)}, this convergence
		also holds with respect to $d_t^{\bar\tau}$. Thus
		$\mathsf S_\tau(\A)$ is dense in the trace-one shell
		\[
		\{\rho\in L^1(\M_\tau,\bar\tau)_+:\bar\tau(\rho)=1\}.
		\]
		
		Finally, the mass functional is $L^1$-continuous, so this shell
		is $L^1$-closed and hence $d_t^{\bar\tau}$-closed. Since
		$L^1(\M_\tau,\bar\tau)_+$ is complete for $d_t^{\bar\tau}$,
		the trace-one shell is complete. Its metric completion is therefore
		exactly \eqref{eq:trace-GNS-states}.
\end{proof}

\begin{remark}
	\label{rem:faithful-densities}
	A faithful density in $L^1(\M,\tau)_+$ exists if and only if
	$\M$ is $\sigma$-finite. Indeed, such a density determines a
	faithful normal positive functional, whereas every
	$\sigma$-finite von Neumann algebra admits a faithful normal state.
	In this case, faithful densities are dense in
	$L^1(\M,\tau)_+$: if $h$ is one fixed faithful density, then
	\[
	A+\varepsilon h\longrightarrow A
	\qquad\text{in }L^1
	\]
	as $\varepsilon\downarrow0$, and hence also in $d_t$.
	The cone of faithful densities is not complete, since its
	$L^1$-closure contains nonfaithful elements. If $\M$ is not
	$\sigma$-finite, then no faithful normal state, and consequently
	no faithful state density, exists.
\end{remark}
	
\subsection{A Haagerup-$L^p$ formulation beyond the semifinite setting}
\label{subsec:Haagerup-formulation}

Let $\M$ be an arbitrary von Neumann algebra. Under the canonical
order isomorphism
\[
L^1(\M)\cong\M_*,
\]
each $\varphi\in\M_*^+$ is represented by a unique density
$h_\varphi\in L^1(\M)_+$.
Following \cite[Section~2 and Appendix~A]{HiaiJencova}, denote by
\[
\operatorname{tr}:L^1(\M)\longrightarrow\C
\]
the canonical linear functional defined by
\[
\operatorname{tr}(h_\varphi):=\varphi(\one),
\qquad
\varphi\in\M_*^+.
\]
This functional should not be confused with the canonical
semifinite trace on the crossed-product algebra used in the
construction of the Haagerup spaces.

For $t>0$, H\"older's inequality in the Haagerup $L^p$-spaces
shows that
\[
h_\varphi^{t/4}h_\psi^{t/2}h_\varphi^{t/4}
\in L^{1/t}(\M)_+,
\]
and hence its $1/t$-power belongs to $L^1(\M)_+$. We may therefore
define
\begin{equation}\label{eq:Haagerup-affinity}
Q_t^{\mathrm{tr}}(\varphi,\psi)
:=
\operatorname{tr}\!\left(
\left(
h_\varphi^{t/4}
h_\psi^{t/2}
h_\varphi^{t/4}
\right)^{1/t}
\right).
\end{equation}
In the notation of \cite{HiaiJencova},
\[
Q_t^{\mathrm{tr}}(\varphi,\psi)
=
Q_{1/2,\,1/t}(\psi\Vert\varphi).
\]
The Haagerup H\"older inequality also gives
\[
0\leq
Q_t^{\mathrm{tr}}(\varphi,\psi)
\leq
\sqrt{\varphi(\one)\psi(\one)}.
\]
Thus one may set
\begin{equation}\label{eq:Haagerup-distance}
d_t^{\mathrm{tr}}(\varphi,\psi)^2
:=
\frac{\varphi(\one)+\psi(\one)}2
-
Q_t^{\mathrm{tr}}(\varphi,\psi).
\end{equation}
When $\M$ is semifinite and
$\varphi=\varphi_A$, $\psi=\varphi_B$ are represented by
$A,B\in L^1(\M,\tau)_+$, these definitions recover
$Q_t^\tau(A,B)$ and $d_t^\tau(A,B)$.

For $0<t\leq2$, the parameters
\[
\alpha=\frac12,
\qquad
z=\frac1t
\]
satisfy
\[
z\geq\max\{\alpha,1-\alpha\}=\frac12.
\]
Hence the data-processing inequality
\cite[Theorem~3.3]{HiaiJencova} gives
\[
Q_t^{\mathrm{tr}}(\Gamma_*\varphi,\Gamma_*\psi)
\geq
Q_t^{\mathrm{tr}}(\varphi,\psi)
\]
for every normal positive unital map
$\Gamma:\N\to\M$. Since $\Gamma_*$ preserves the masses of
positive functionals, it follows that
\[
d_t^{\mathrm{tr}}(\Gamma_*\varphi,\Gamma_*\psi)
\leq
d_t^{\mathrm{tr}}(\varphi,\psi).
\]
In particular, this contractivity holds for $1\leq t\leq2$,
the range in which Theorem~\ref{thm:classification}\textup{(i)}
establishes metricity in the semifinite setting.

The finite spectral-corner argument used in the proof of
Theorem~\ref{thm:classification}\textup{(i)} relies on a fixed
semifinite trace and does not transfer directly to the homogeneous
subspaces defining the Haagerup $L^p$-spaces. A natural route to an
analogue of the three-point positivity theorem for arbitrary von
Neumann algebras is Haagerup's reduction to finite tracial algebras;
see \cite{HaagerupJungeXu}. We do not pursue that reduction here.

\section{Acknowledgement}

We appreciate Prof. Tamotsu Teruya for a fruitful discussion in Section 7. 
The research was partially supported by the KAKENHI Grant Number JP25K07036.


\end{document}